\documentclass[11pt]{amsart}
\usepackage{amsmath,amsfonts,amssymb,amsthm}
\usepackage{graphics}
\usepackage{epsfig,psfrag}
\usepackage{comment}

\newtheoremstyle{thm}
  {9pt}{9pt}{\itshape}{}{\bfseries}{}{.5em}{}

\theoremstyle{thm}
\newtheorem{thm}{Theorem}[section]

\newtheorem{lemma}[thm]{Lemma}
\newtheorem{prop}[thm]{Proposition}
\newtheorem{conj}[thm]{Conjecture}

\newtheorem{exa}[thm]{Example}

\newtheoremstyle{defin}
  {9pt}{9pt}{}{}{\bfseries}{}{.5em}{}
\theoremstyle{defin}

\newtheorem{cl}{Claim}
\newtheoremstyle{exm}
  {9pt}{9pt}{}{}{\scshape}{}{.5em}{}
\theoremstyle{exm}

\newtheoremstyle{proof}
  {}{}{}{}{\itshape}{:}{.5em}{}
\theoremstyle{proof}

\newcommand{\set}[1]{\{#1\}}

\newcommand{\R}{{\mathbb R}}

\newcommand{\s}[1]{\mathbf{#1}}

\DeclareMathOperator{\vol}{vol}
\DeclareMathOperator{\cat}{Cat}
\DeclareMathOperator{\cone}{Cone}

\DeclareMathOperator{\ia}{ia}
\DeclareMathOperator{\ea}{ea}
\DeclareMathOperator{\Inv}{Inv}

\def\emp{\nothing}
\def\sq{\square}

\def\rr{\mathbb R}

\def\cC{\mathcal C}

\def\cG{\mathcal G}

\def\bP{\mathbf{P}}
\def\ssu{\subset}

\def\<{\langle}
\def\>{\rangle}

\def\rT{{\text {\rm T} } }

\def\0{{\mathbf 0}}

\def\nothing{\varnothing}

\def\.{\hskip.06cm}
\def\ts{\hskip.03cm}

\def\vol{{\text {\rm vol}}}

\def\pC{{{\text{\ts\bf{C}}}}}
\def\pG{{{\text{\ts\bf{G}}}}}
\def\pD{{{\text{\ts\bf{D}}}}}
\def\pT{{{\text{\ts\bf{T}}}}}
\def\pS{{{\text{\ts\bf{S}}}}}
\def\pO{{{\text{\ts\bf{O}}}}}

\def\rZ{ {\text {\rm Z} } }
\def\rT{ {\text {\rm T} } }
\def\IP{\Inv}

\title[Cayley and Tutte polytopes]{Triangulations of Cayley and Tutte polytopes}

\author[Matja\v z~Konvalinka]{ \ Matja\v z~Konvalinka$^\star$}

\author[Igor~Pak]{ \ Igor~Pak$^\dagger$}

\thanks{\thinspace ${\hspace{-.45ex}}^\star$Department of Mathematics, University of Ljubljana, 1000 Ljubljana, Slovenia.}

\thanks{\thinspace ${\hspace{-.45ex}}^\dagger$Department of Mathematics, UCLA, Los Angeles, CA 90095, USA}

\date{June 1, 2011}

\begin{document}

\begin{abstract}
\emph{Cayley polytopes} were defined recently 
as convex hulls of \emph{Cayley compositions} introduced by Cayley in 1857. 
In this paper we resolve \emph{Braun's conjecture},
which expresses the volume of Cayley polytopes in terms of the number of connected graphs.
We extend this result to two one-variable deformations of Cayley polytopes (which we call
\emph{$t$-Cayley} and \emph{$t$-Gayley polytopes}), and to the most general two-variable deformations, which
we call \emph{Tutte polytopes}.  The volume of the latter is given via an evaluation of the
\emph{Tutte polynomial} of the complete graph.

Our approach is based on an explicit triangulation of the Cayley and Tutte polytope.
We prove that simplices in the triangulations correspond to labeled trees. 
The heart of the proof is a direct bijection based on the
\emph{neighbors-first search} graph traversal algorithm.
\end{abstract}

\maketitle

\section{Introduction} \label{intro}

In the past several decades, there has been an explosion in the
number of connections and applications between Geometric and
Enumerative Combinatorics.  Among those, a number of new
families of ``combinatorial polytopes'' were discovered, whose
volume has a combinatorial significance.
Still, whenever a new family of $n$-dimensional polytopes is discovered
whose volume is a familiar integer sequence (up to scaling), it feels
like a ``minor miracle'', a familiar face in a crowd in a foreign country,
a natural phenomenon in need of an explanation.

\smallskip

In this paper we prove a surprising conjecture due to Ben Braun~\cite{BBL},
which expresses the volume of the Cayley polytope in terms of the number
of connected labeled graphs.  Our proof is robust enough to allow
generalizations in several directions, leading to the definition
of \emph{Tutte polytopes}, and largely explaining this latest
``minor miracle''.

\smallskip

We start with the following classical result.

\begin{thm}[Cayley, 1857] \label{t:cayley}
The number of integer sequences $(a_1,\ldots,a_n)$ such that
$1\le a_1 \le 2$, and $1\le a_{i+1} \le 2 \ts a_i$ for $1\le i <n$, is equal
to the total number of partitions of integers $N \in \{0,1,\ldots,2^{n}-1\}$
into parts $1,2,4,\ldots,2^{n-1}$.
\end{thm}

Although Cayley's original proof~\cite{Cay} uses only elementary generating
functions, it inspired a number of other proofs and variations~\cite{APRS,BBL,CLS,KP}.
It turns out that Cayley's theorem is best understood in a geometric setting,
as an enumerative problem for the number of integer points in an
$n$-dimensional polytope defined by the inequalities as in the theorem.

\smallskip

Formally, following~\cite{BBL}, define the \emph{Cayley polytope} $\pC_n\ssu \rr^n$ by inequalities:
$$
1 \. \leq \. x_1 \.\leq 2\ts , \ \ \text{and} \ \. 1 \. \leq \. x_i \. \leq \. 2 \ts x_{i-1} \,
\mbox{ for } \ i \ts = \ts 2,\ldots,n\ts,
$$
so that the number of integer points in $\pC_n$ is the number of integer
sequences $(a_1,\ldots,a_n)$, and the number of certain partitions, as in Cayley's theorem.

\smallskip

In~\cite{BBL}, Braun made the following interesting conjecture about the volume
of~$\pC_n$.  Denote by~$\cC_n$ the set of connected graphs on~$\ts n \ts $
nodes\footnote{To avoid ambiguity, throughout the paper, we distinguish
\emph{graph nodes} from \emph{polytope vertices}.},
and let $C_n = \bigl|\cC_n\bigr|$.

\begin{thm}[Formerly Braun's conjecture] \label{t:pol}
Let $\pC_n \ssu \rr^n$ \. be the Cayley polytope defined above.
Then \ts $\vol \ts \pC_n = C_{n+1}/n!$.
\end{thm}

This result is the first in a long chain of results we present
in this paper, leading to the following general result. Let $0 < q \leq 1$ and $t \geq 0$.
Define the \emph{Tutte polytope} $\pT_n(q,t)\ssu \rr^n$ by inequalities:
$x_n \geq 1-q$ and
$$ (\diamond) \hskip.8cm
q \ts x_i \, \leq \, q \ts (1 + t) \ts x_{i-1} \. - \. t\ts (1-q)(1-x_{j-1})\ts,
$$
where $1 \leq j \leq i \leq n$ and $x_0 = 1$.

\begin{thm}[Main result] \label{t:main}
Let $\pT_n(q,t) \ssu \rr^n$ \. be the Tutte polytope defined above.
Then
$$\vol \ts \pT_n(q,t) \, = \, t^n \rT_{K_{n+1}}(1+q/t,1+t)\ts / n!,
$$
where $\rT_H(x,y)$ denotes the \emph{Tutte polynomial} of graph~$H$.
\end{thm}

One can show that in certain sense, Tutte polytopes are a
two variable deformation of the Cayley polytope:
$$
\lim_{q\to 0+} \. \pT_n(q,1) \, = \, \pC_n\ts.
$$
To see this, note that for $t = 1$, the inequalities with $j=1$ in $(\diamond)$ give $x_i \leq 2x_{i-1}$, and for $j > 1$, we get $x_{j-1} \geq 1$ as $q \to 0+$.

\smallskip

Now, recall that $\rT_H(1,2)$ is the number of connected subgraphs
of~$H$, a standard property of Tutte polynomials (see e.g.~\cite{Bol}).
Letting $q\to 0+$ and $t=1$ shows that Theorem~\ref{t:main} follows
immediately from Theorem~\ref{t:pol}.  In other words, our main theorem
is an advanced generalization of Braun's Conjecture (now Theorem~\ref{t:pol}).

\smallskip

The proof of both Theorem~\ref{t:pol} and~\ref{t:main} is based
on explicit triangulations of polytopes.  The simplices in the
triangulations have a combinatorial nature, and are in bijection
with labeled trees (for the Cayley polytope) and forests (for the
Tutte polytope) on $n+1$ nodes.  This bijection is based
on a variant of the \emph{neighbors-first search} (NFS)
graph traversal algorithm studied by Gessel and Sagan~\cite{GS}.
Roughly speaking, in the case of Cayley polytopes, the volume of
a simplex in bijection with a labeled tree~$T$ corresponds to the set
of labeled graphs for which~$T$ is the output of the~NFS.

\smallskip

To be more precise, our most general construction gives two subdivisions of the Tutte polytope,
a triangulation (subdivision into simplices) and a coarser subdivision that
can be obtained from simplices with products and coning. Some (but not all) of the simplices
involved are \emph{Schl\"{a}fli orthoschemes} (see below).
The polytopes in the coarser subdivision are in bijection with plane forests,
so there are far fewer of them. In both subdivisions, the volume
of the simplex or the polytope in bijection with a forest~$F$ on $n+1$ nodes,
times~$n!$, is equal to the generating function of all the graphs $G$ that
map into it by the number of connected components (factor $q^{k(G)-1}$)
and the number of edges (factor $t^{|E(G)|}$).

\smallskip

Rather than elaborate on the inner working of the proof, we illustrate
the idea in the following example.

\begin{exa} {\small \rm
The triangulation of $\rT_2(q,t)$ is shown on the
left-hand side of Figure~\ref{fig16}. For example, the top triangle is labeled by
the tree with edges $12$ and $13$; its area, multiplied by $2!$, is $t^2(1+t)$,
and it also has two graphs that map into it, the tree itself (with two edges)
and the complete graph on $3$ nodes (with three edges). The coarser subdivision
is shown on the right-hand side of Figure~\ref{fig16}.
The bottom rectangle corresponds to the plane forest with two components,
the first having two nodes. Its area, multiplied by $2!$, is $2qt$,
and there are indeed two graphs that map into it, both with two components
(and hence a factor of $q$) and one edge (and hence a factor of $t$).
Triangulation of $\rT_3(q,t)$ is shown in Figure~\ref{fig18}.

\begin{figure}[ht!]
 \begin{center}
   \includegraphics[height=7.8cm]{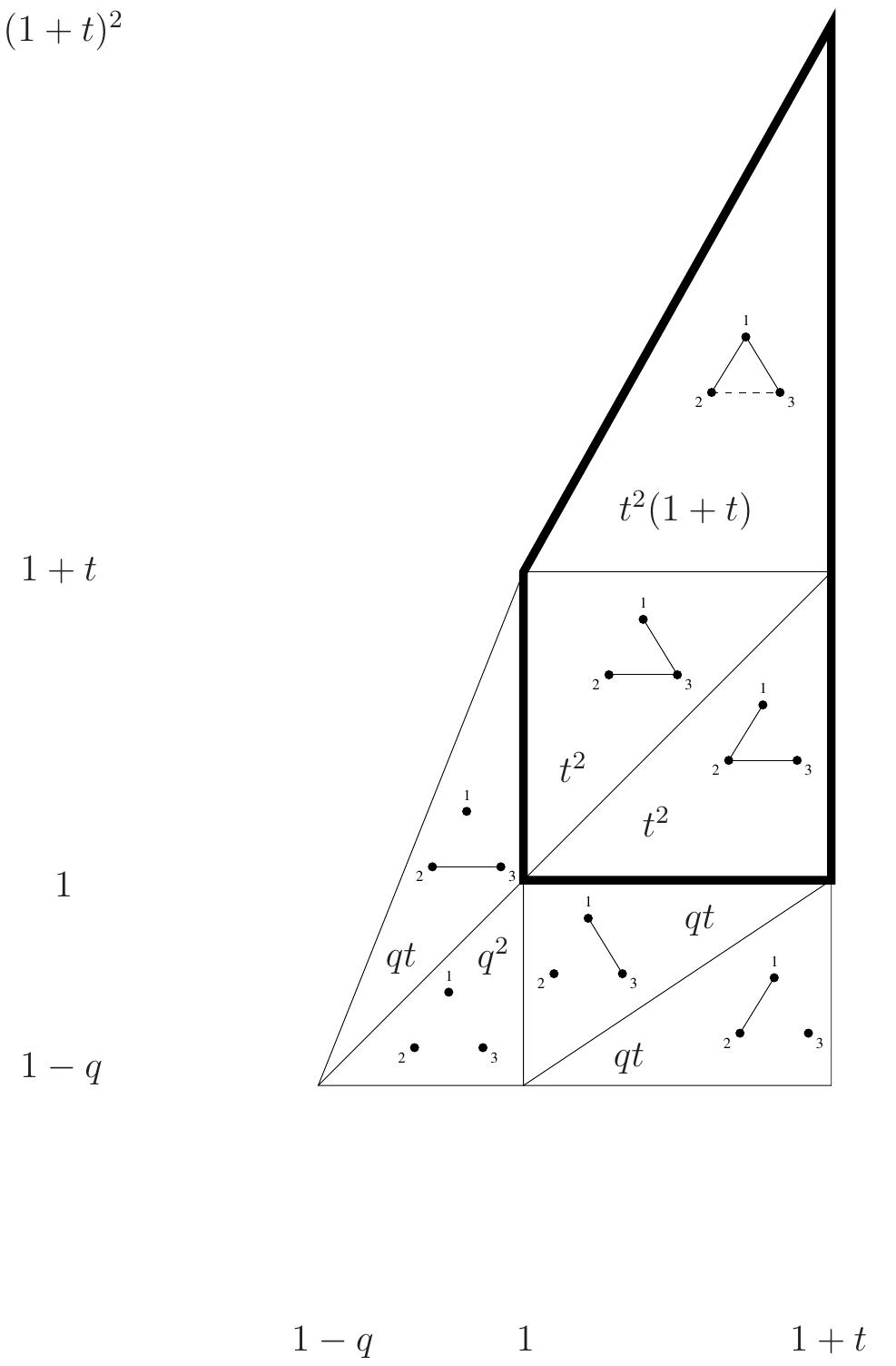}
   \qquad \qquad \qquad \qquad
   \includegraphics[height=7.8cm]{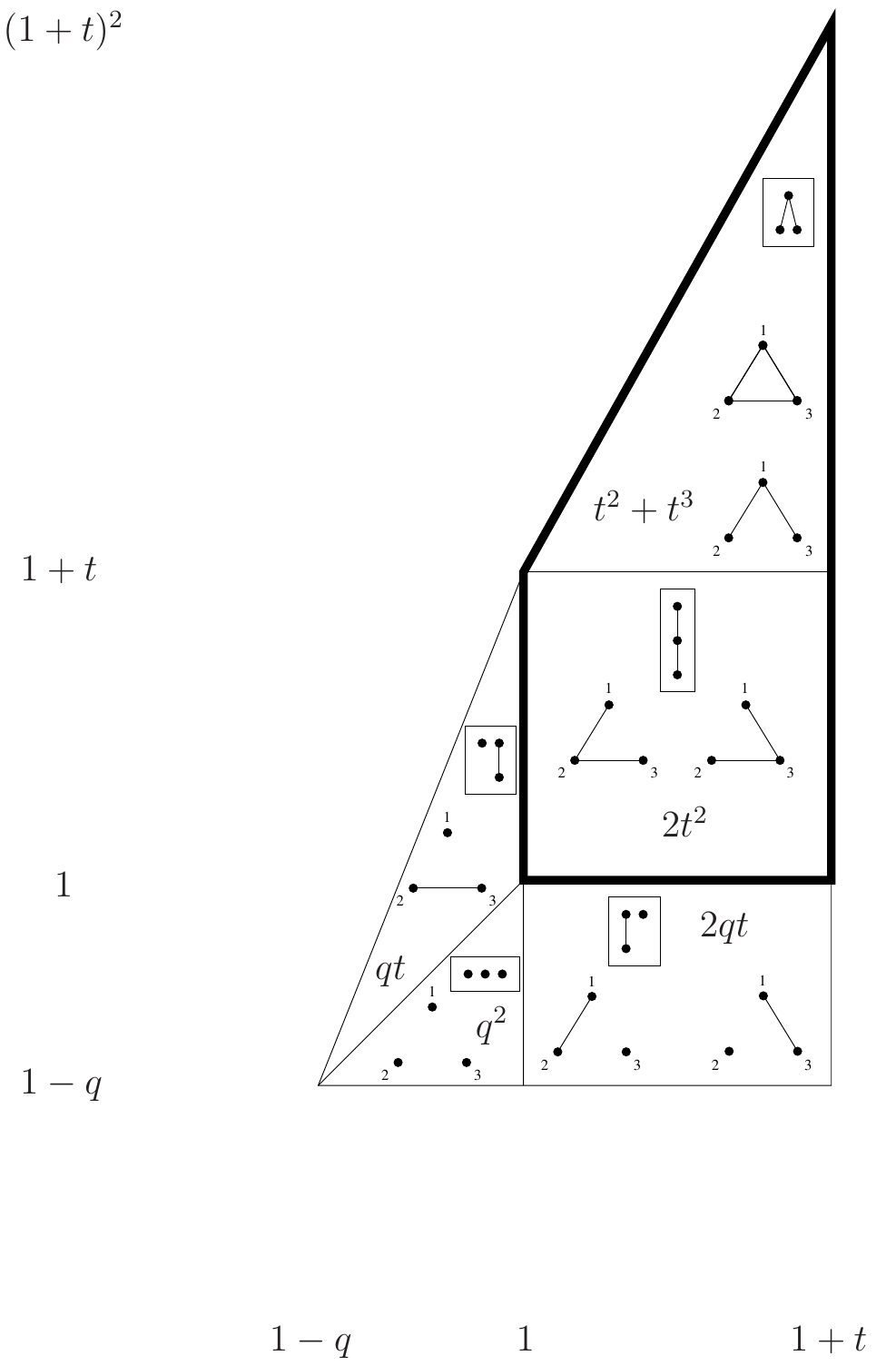}
   \caption{A triangulation and a subdivision of the Tutte polytope $\s T_2(q,t)$.}
   \label{fig16}
 \end{center}
\end{figure}

\begin{figure}[ht!]
 \begin{center}
   \includegraphics[height=5cm]{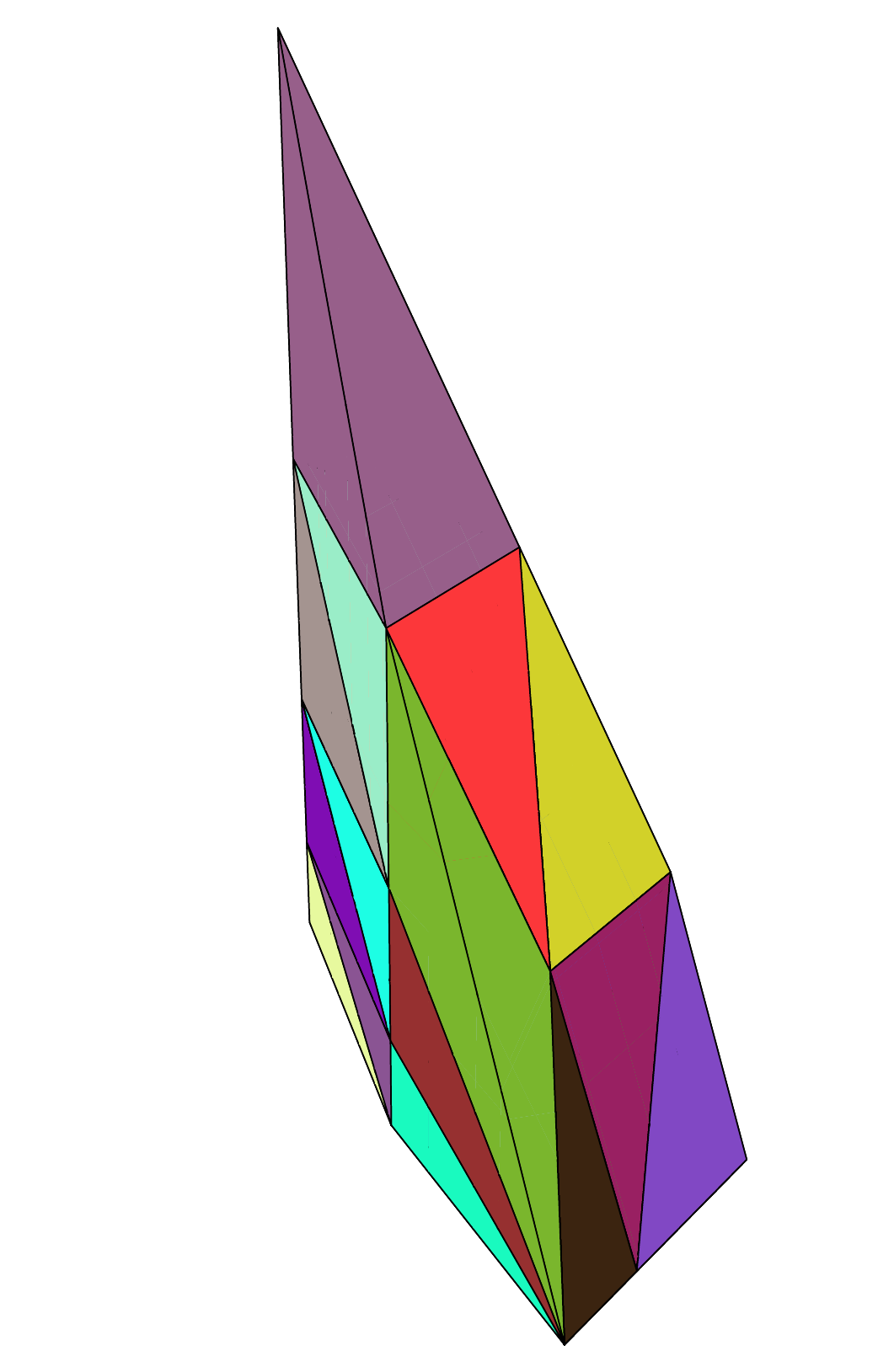}
   \qquad \qquad
   \includegraphics[height=5cm]{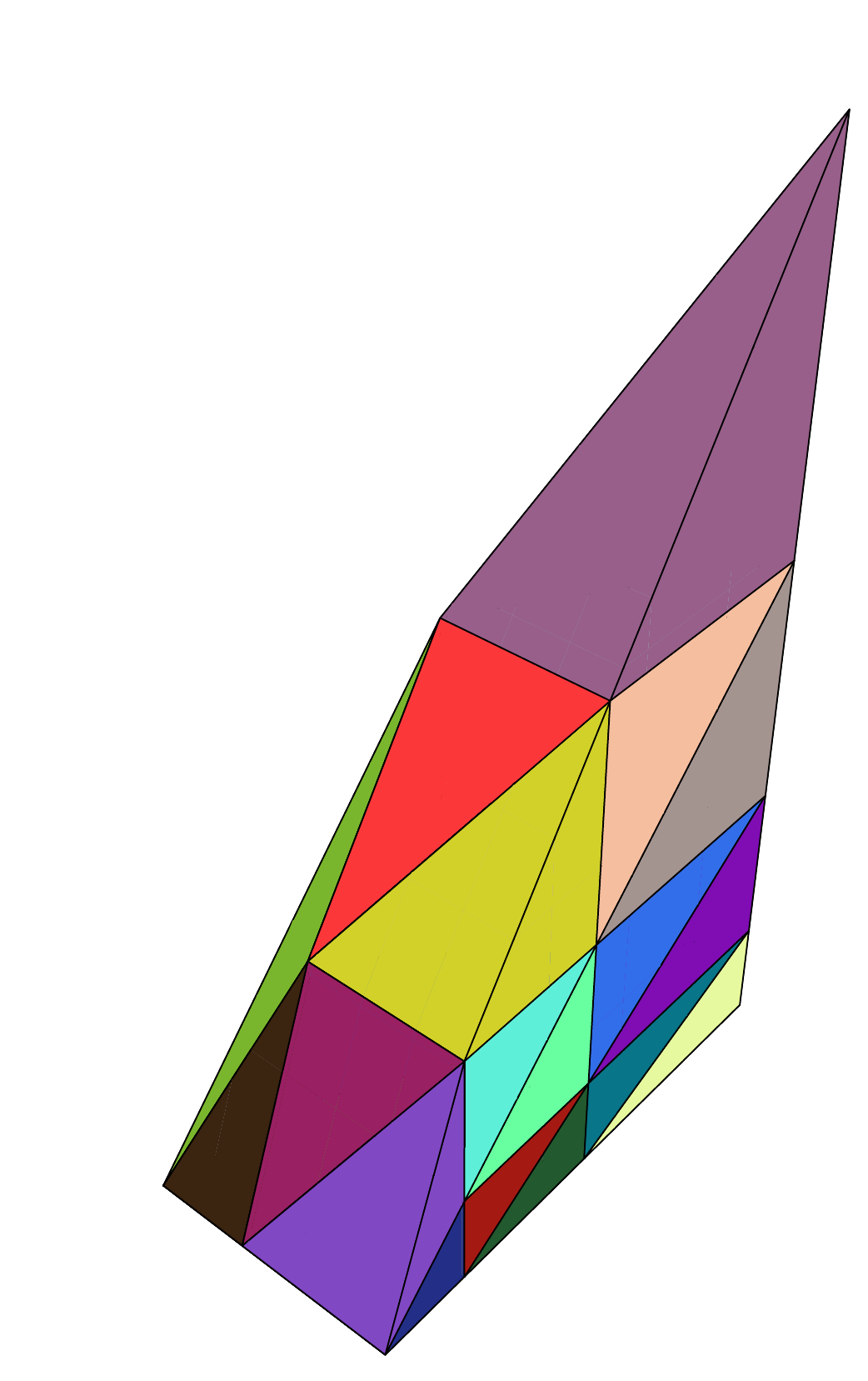}
   \caption{A triangulation of the Tutte polytope $\s T_3(q,t)$ from two angles.}
   \label{fig18}
 \end{center}
\end{figure}

}

\end{exa}

The rest of the paper is structured as follows.  We begin
with definitions and basic combinatorial results in Section~\ref{s:comb}.
In Sections~\ref{triang} and~\ref{another} we construct a triangulation
and a coarse subdivision of the Cayley polytope.  In Section~\ref{gayley} we
present a similar construction for what we call the \emph{Gayley polytope},
which can be defined as a special case of the Tutte polytope $\s T_n(1,1)$.
Two one parametric families of deformations of Cayley and Gayley polytopes
are then considered in Section~\ref{t}; we call these \emph{$t$-Cayley}
and \emph{$t$-Gayley polytopes}.  \emph{Tutte polytopes} are then defined
and analyzed in Section~\ref{s:tutte}. The vertices of the polytopes are
studied in Sections~\ref{vertices}. An ad hoc application of the volume of $t$-Cayley
polytopes to the study of \emph{inversion polynomials} is given in
Section~\ref{s:app}.  We illustrate all constructions with examples in Section \ref{examples}.
 The proofs of technical results in
Sections~\ref{triang}$-$\ref{vertices} appear in the lengthy
Section~\ref{proofs}.  We conclude with final remarks and open
problems in Section~\ref{s:fin}.

\medskip

\section{Combinatorial and geometric preliminaries}\label{s:comb}

\subsection{}\label{ss:comb-1} \.
A \emph{labeled tree} is a connected acyclic graph. We take each labeled tree to be rooted
at the node with the maximal label. A \emph{labeled forest} is an acyclic graph. Its components
are labeled trees, and we root each of them at the node with the maximal label. \emph{Cayley's formula}
states that there are $n^{n-2}$ labeled trees on $n$ nodes.
An \emph{unlabeled plane forest} is a graph without cycles in which
we do not distinguish the nodes, but we choose a root in each component,
which is an \emph{unlabeled plane tree}, and the subtrees at any node,
as well as the components of the graph, are linearly ordered (from left to right).
The number of plane forests on $n$ nodes is the $n$-th Catalan number
$\cat(n) = \frac 1{n+1} \binom{2n}n$, and the number of plane tree on $n$ nodes
is $\cat(n-1)$. The \emph{degree} of a node in a
plane forest is the number of its successors, which is the usual (graph) degree
if the node is a root, and one less otherwise. The \emph{depth-first traversal}
goes through the forest from the left-most tree to the right; within each tree,
it starts at the root, and if nodes $v$ and $v'$ have the same parent and $v$
is to the left of $v'$, it visits $v$ and its successors before $v'$.

\smallskip

The \emph{degree sequence} of a tree $T$ on $n$ nodes is the sequence
$(d_1,\ldots,d_n)$ where $d_i$ is the degree of the $i$-th node in depth-first traversal.
Since the last node is a leaf, the degree sequence always ends with a zero.
The degree sequence determines the plane tree uniquely, and we have
$\sum_{i=1}^n d_i = n - 1$. The degree sequence of a forest $F$ is the
concatenation of the degree sequences of its components, and it determines
the plane forest uniquely. Finally, if we erase zeros marking the ends of components,
we get a \emph{reduced degree sequence}.
We refer to \cite[\S~5.3 and Exc.~6.19e]{Stanley} for further details.

\subsection{}\label{ss:comb-2} \.
For a (multi)graph $G$ on the set of nodes $V$, denote by $k(G)$ the number
of connected components of $G$, and by $e(G)$ the number of edges of $G$.
Consider  a polynomial
$$
\rZ_G(q,t) \. = \. \sum_{H \subseteq G} q^{k(H)-k(G)} t^{e(H)},$$
where the sum is over all spanning subgraphs $H$ of $G$.
This polynomial is a statistical sum in the \emph{random cluster model}
in statistical mechanics. It  is related to the \emph{Tutte polynomial}
$$
\rT_G(x,y) \. = \. \sum_{H \subseteq G} (x-1)^{k(H)-k(G)} (y-1)^{e(H)-|V|+k(H)}
$$
by the equation
$$
\rT_G(x,y) \. = \. (y-1)^{k(G)-|V|} \. \rZ_G((x-1)(y-1),y-1)\ts.
$$
Tutte's classical result is a combinatorial interpretation for coefficients of
the Tutte polynomial~\cite{Tutte}.  He showed that for a connected graph~$G$ we have:
$$
(\lozenge) \ \ \ \rT_G(x,y) \. = \. \sum_{T\in G} \ts x^{\ia(T)}\ts y^{\ea(T)}\.,
$$
where the summation is over all spanning trees~$T$ in~$G$; here $\ia(T)$ and $\ea(T)$
denote the number of \emph{internally active} and \emph{externally active}
edges in~$T$, respectively.  While both $\ia(T)$ and $\ea(T)$  depend on
the ordering of the edges in~$G$, the sum $(\lozenge)$ does not
(see \cite[\S X.5]{Bol} for definitions and details).

\smallskip

For the complete graph~$K_n$, the Tutte polynomial and its evaluations are
well studied (see~\cite{Tutte,Ges2}).  In this case, under a lexicographic
ordering of edges, the statistics $\ia(T)$ and $\ea(T)$ can be interpreted
combinatorially~\cite{Ges2,GS} via the \emph{neighbor-first search} (NFS) introduced
in~\cite{GS}, a variant of which is also crucial for our purposes. Take a labeled connected graph $G$ on $n+1$ nodes.
Choose the node with the maximal label, i.e.\ $n+1$, as the first active node
(and also the $0$-th visited node). At each step, visit the previously
unvisited neighbors of the active node in decreasing order of their labels,
and make the one with the smallest label the new active node.\footnote{Note that
in~\cite{GS}, the NFS starts at the node with the minimal label, and the neighbors of
the active node are visited in \emph{increasing} order of their labels.}
If all the neighbors of the active node have been visited,
backtrack to the last visited node that has not been an active node,
and make it the new active node. The resulting search tree $T$ is a
labeled tree on $n+1$ nodes, we denote it $\Phi(G)$ (see Example~\ref{exm1}).

\smallskip

In a special case, the polynomial $\IP_n(y)=\rT_{K_n}(1,y) \ts y^{1-n}$ is the 
classical \emph{inversion polynomial}~\cite{MR} (see also~\cite{Ges1,GW,GJ2}), 
a generating function for the number of spanning trees with respect to
inversions.

\subsection{}\label{ss:comb-tri} \. Let $\bP \ssu \rr^n$ be a convex polytope.
A \emph{triangulation} of~$\bP$ is a \emph{dissection} of~$\bP$ into $n$-simplices.
Throughout the paper, all triangulations are in fact \emph{polytopal subdivisions};
we do not emphasize this as this follows from their explicit construction.  We refer
to~\cite{DRS} for a comprehensive study of triangulations of convex polytopes.

Denote by $\pO(\ell_1,\dots,\ell_n) \ssu \rr^n$ a simplex defined as convex hull
of vertices
$$
(0,0,0,\ldots,0), \ (\ell_1,0,0,\ldots,0), \ (\ell_1,\ell_2,0,\ldots,0), \, \ldots \,,
\ (\ell_1,\ell_2,\ell_3\ldots,\ell_n)\ts.
$$
Such simplices, and the polytopes we get if we permute and/or translate the coordinates, are called \emph{Schl\"{a}fli orthoschemes}, or \emph{path-simplices}
(see Subsection~\ref{s:fin-2}).
Obviously, $\vol \pO(\ell_1,\ldots,\ell_n)= \ell_1\cdots \ell_n/n!$.

\medskip

\section{A triangulation of the Cayley polytope} \label{triang}

Attach a \emph{coordinate} of the form $x_i/2^{j}$ to each node of the tree $T$ rooted at the node with label $n+1$, where $i$ is the position of the node in the NFS, and $j$ is a non-negative integer defined as follows. Attach $x_0$ to the root; and if the node $v$ has coordinate $x_i/2^j$ and successors $v_1,\ldots,v_k$ (in increasing order of their labels), then make the coordinates of $v_k,\ldots,v_1$ to be $x_{i'}/2^j,x_{i'+1}/2^{j+1},\ldots,x_{i'+k-1}/2^{j+k-1}$. See Figure~\ref{fig7} for an example.

\smallskip

Define $\alpha(T) = \sum_i j_i$. For the next lemma, which gives another characterization of $\alpha(T)$, note first that in a rooted labeled tree (as well as in a plane tree), we have the natural concept of an \emph{up} (respectively, \emph{down}) step, i.e.\ a step from a node to its parent (respectively, from a node to its child), as well as a \emph{down right} step, i.e.\ a down step $v \to v''$ that follows an up step $v' \to v$ so that $v''$ has a larger label than (or is the the right of) $v'$. Call a path of length $k \geq 2$ in a rooted labeled tree (or a plane tree) a \emph{cane path} if the first $k-1$ steps are up and the last one is down right (see Figure~\ref{fig14}).

 \begin{figure}[ht!]
 \begin{center}
   \includegraphics[height=2.6cm]{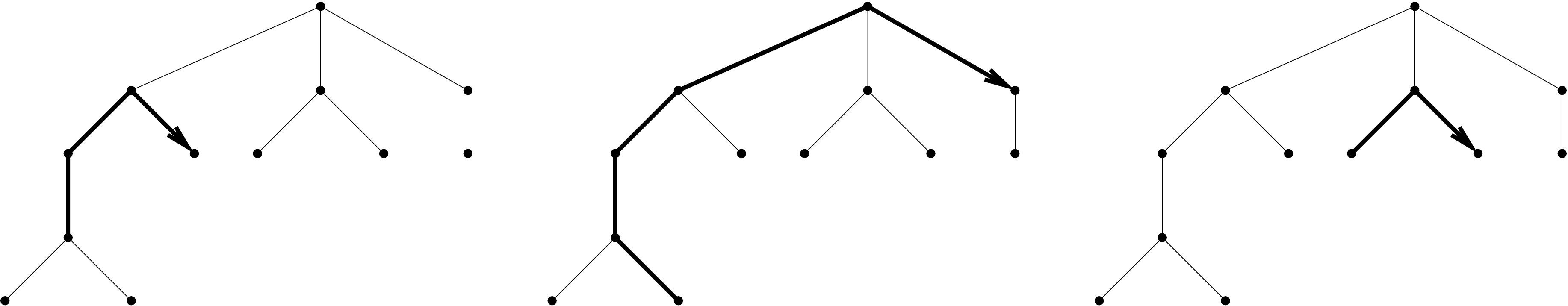}
   \caption{Cane paths in a tree.}
   \label{fig14}
 \end{center}
\end{figure}

\begin{lemma} \label{alpha}
 For a node $v$ with coordinate $x_i/2^{j}$, $j$ is the number of cane paths in $T$ that start in $v$. In particular, $\alpha(T)$ is the number of cane paths in $T$.
\end{lemma}

Arrange the coordinates of the nodes $1,\ldots,n$ according to the labels. More precisely, define
$$\pS_T \. = \.
\set{(x_1,\ldots,x_n) \colon 1 \leq x_{i_1}/2^{j_1} \leq x_{i_2}/2^{j_2} \leq \ldots \leq x_{i_n}/2^{j_n} \leq 2}\ts,
$$
where the coordinate of the node with label $k$ is $x_{i_k}/2^{j_k}$. 
Note that $\pS_T$ is a Schl\"{a}fli orthoscheme with parameters $2^{j_1},\ldots,2^{j_n}$ (see Example~\ref{exm2}).

\begin{thm} \label{thm2}
 For every labeled tree $T$ on $n+1$ nodes, the set $\pS_T$ is a simplex, and
 $$
 n! \ts \vol \pS_T \. = \. |\set{G \in \cC_{n+1}, \, \, \mbox{\textup{s.t.}} \ \Phi(G) = T}| \.
 = \. 2^{\alpha(T)}.
 $$
 Furthermore, simplices $\pS_T$ triangulate the Cayley polytope $\pC_n$.
 In particular,
 $$
 n! \ts \vol \pC_n \. = \. |\cC_{n+1}|\ts.
 $$
\end{thm}

The theorem is proved in Section~\ref{proofs}. Note that Theorem~\ref{thm2} implies
Braun's Conjecture (Theorem~\ref{t:pol}). Figure~\ref{fig8} shows two views of
the resulting triangulation of $\pC_3$.

\begin{figure}[ht!]
 \begin{center}
 \includegraphics[height=5cm]{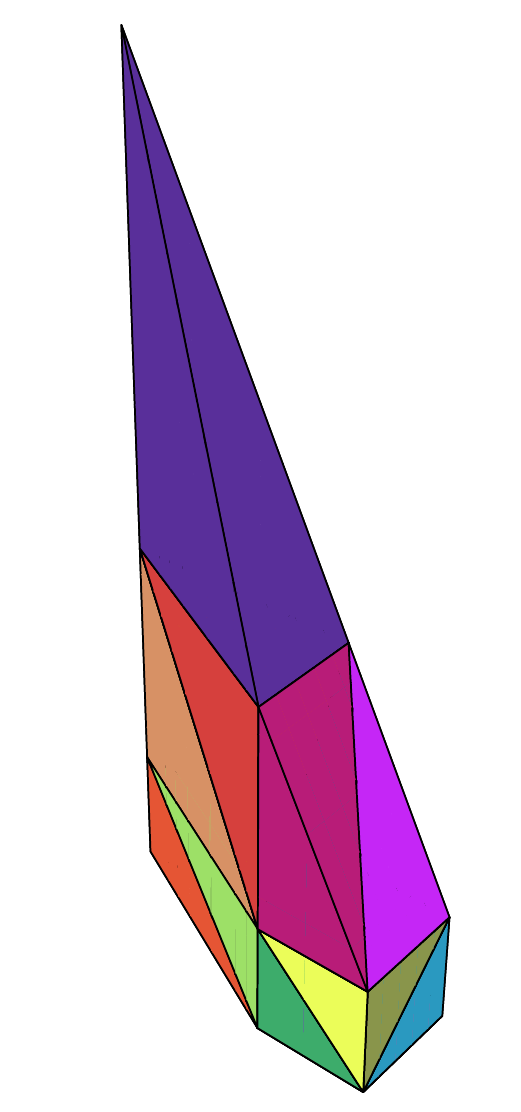}
 \qquad \qquad
 \includegraphics[height=5cm]{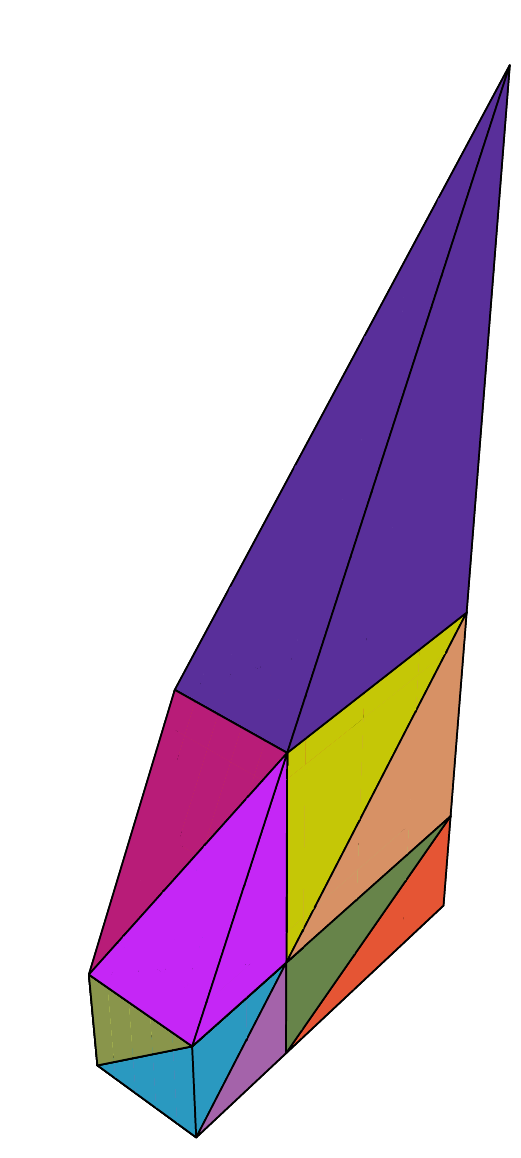}
   \caption{A triangulation of $\pC_3$ from two different angles.}
   \label{fig8}
 \end{center}
\end{figure}

\medskip

\section{Another subdivision of the Cayley polytope} \label{another}

The triangulation of the Cayley polytope described in the previous section
proves Braun's Conjecture by dividing the Cayley polytope into $(n+1)^{n-1}$
simplices. In this section we show how to subdivide the Cayley polytope into a
much smaller number, $\cat(n)$, of polytopes, each a direct product of orthoschemes.
Potentially of independent interest, this constructions paves a way to prove
Theorem~\ref{thm2}.

\smallskip

Start by erasing all labels (but not the coordinates) from the labeled tree $\Phi(G)$,
to make it into a plane tree $\Psi(G)$. For each node $v$ of a plane tree $T$ with successors with
coordinates $x_{i}/2^j,x_{i+1}/2^{j+1},\ldots,x_{i+k-1}/2^{j+k-1}$, take inequalities
$$
1 \. \leq \. x_{i+k-1}/2^{j+k-1} \. \leq \. \ldots \. \leq x_{i+1}/2^{j+1} \. \leq x_{i}/2^j \.
\leq 2\ts.
$$
Equivalently, take inequalities
$$\begin{array}{lclcl}
 2^j &\!\!\!\leq\!\! &x_i & \!\!\!\leq\!\! &2^{j+1}\\
 2^{j+1} &\!\!\!\leq\!\! &x_{i+1} & \!\!\!\leq\!\! &2x_i\\
&&\vdots&& \\
  2^{j+k-1} &\!\!\!\leq\!\! & x_{i+k-1} & \!\!\!\leq\!\! & 2x_{i+k-2}.
\end{array}
$$
%
Denote the resulting polytope $\pD_T$ (see Example~\ref{exm3}).

\begin{thm} \label{thm3}
 For every plane tree $T$ on $n+1$ nodes, the set $\pD_T$ is a bounded polytope, and
$$n!\ts\vol \pD_T \. = \. |\set{G \in \cC_{n+1} \, \,  \mbox{\textup{s.t.}} \ \Psi(G) = T}|
\. = \. 2^{\binom{n+1}2-\sum_{i=1}^{n+1} id_i} \binom n{d_1,d_2,\ldots},
$$
 where $(d_1,\ldots,d_{n+1})$ is the degree sequence of $T$.
Furthermore, polytopes $\pD_T$ form a subdivision of the Cayley polytope $\pC_n$. In particular,
 $$n!\ts\vol \pC_n \. = \. |\set{G \in \cC_{n+1}}| \. = \. C_{n+1}.$$
\end{thm}

Figure~\ref{fig9} shows two views of the resulting subdivision of~$\pC_3$.

\begin{figure}[ht!]
 \begin{center}
    \includegraphics[height=5cm]{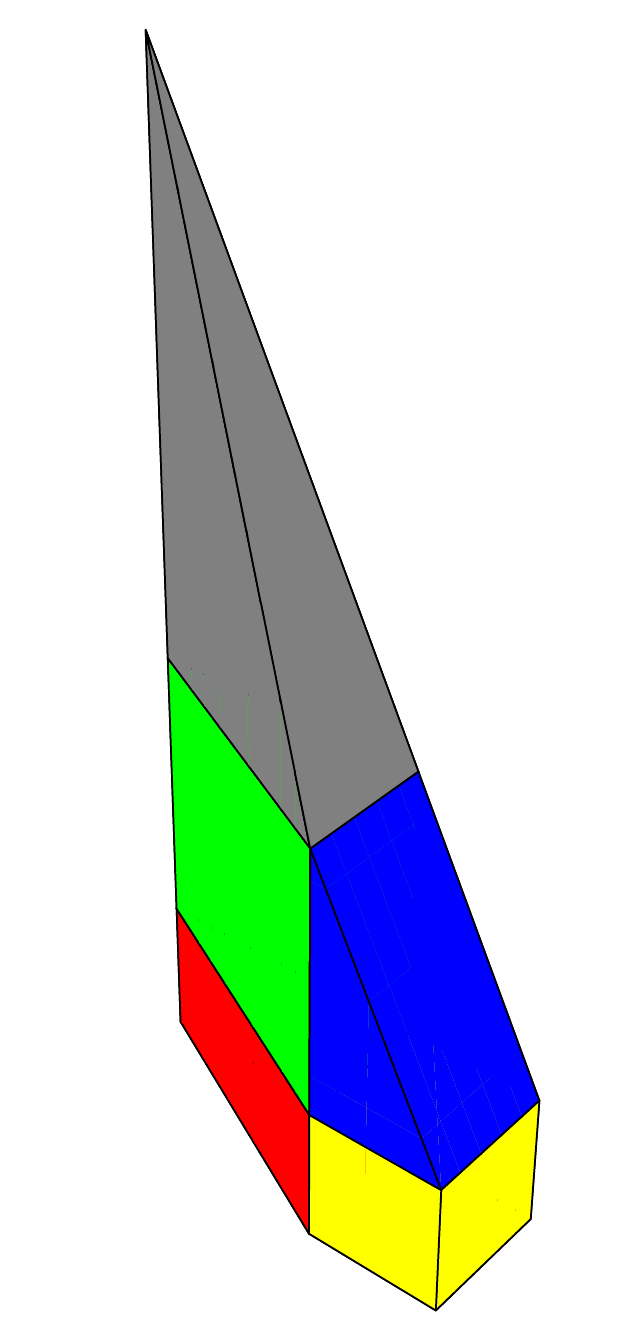}
   \qquad \qquad
   \includegraphics[height=5cm]{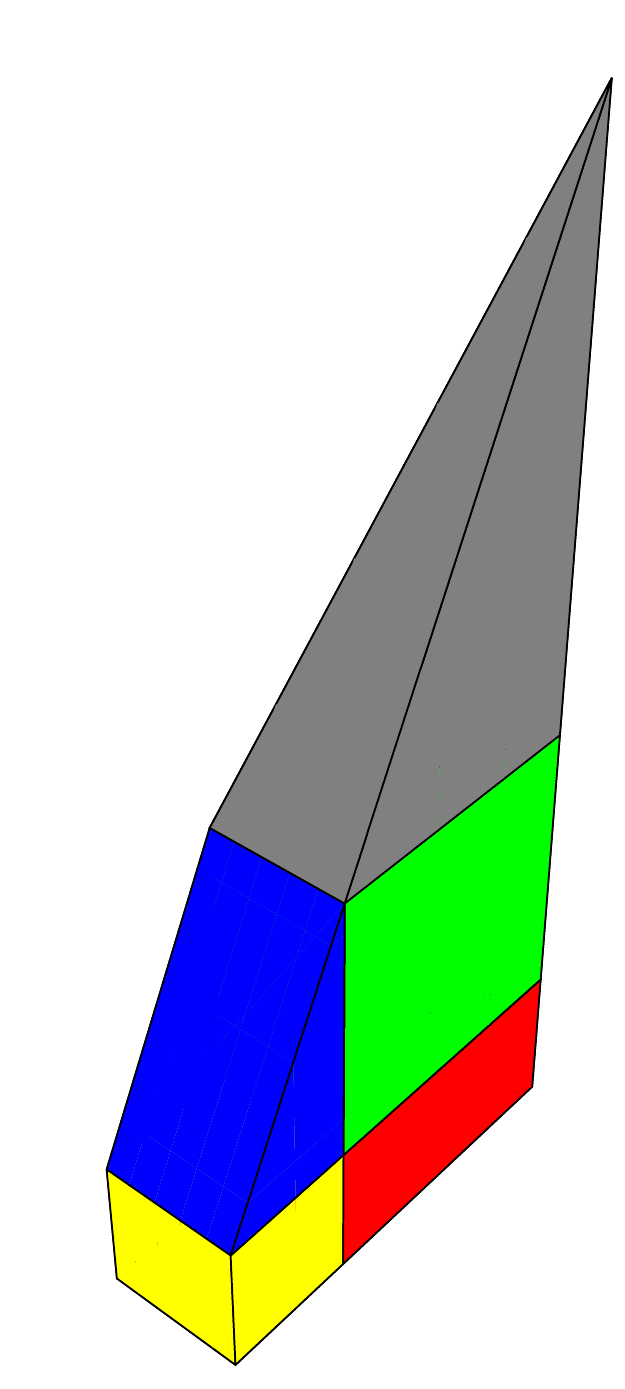}
   \caption{A subdivision of $\pC_3$ from two different angles.}
   \label{fig9}
 \end{center}
\end{figure}

\medskip

\section{The Gayley polytope} \label{gayley}

In this section we introduce the
\emph{Gayley\footnote{Charles Mills Gayley (1858 -- 1932), was a professor of
English and Classics at UC Berkeley; the Los Angeles street on which much of
this research was done is named after him.} polytope} $\pG_n$ which contains
the Cayley polytope~$\pC_n$ and whose volume corresponds to \emph{all} labeled graphs,
not just connected graphs.

Denote by $\cG_n$ the set of labeled graphs on~$n$ nodes.  Obviously, $\cC_n \ssu \cG_n$ and $|\cG_n|=2^{\binom{n}{2}}$.
Replace the $1$'s by $0$'s on the left-hand side of the inequalities defining the Cayley polytope; namely, define
$$\pG_n = \set{(x_1,\ldots,x_n) \colon 0 \leq x_1 \leq 2, 0 \leq x_i \leq 2\ts x_{i-1} \mbox{ for } i = 2,\ldots,n}$$
Note that $\pG_n$ is a Schl\"{a}fli orthoscheme, $\pG_n = \pO(2,4,\ldots,2^n)$, and has volume $2^{\binom{n+1}2}/n!$.
In other words,
$$n! \ts \vol \pG_n \. = \. |\set{G \in \cG_{n+1}}|\ts.$$
Extending the construction in Section~\ref{triang}, we give an explicit triangulation of~$\pG_n$ with simplices
corresponding to labeled forests on $(n+1)$ nodes.  This triangulation will prove useful later.

\smallskip

Start with an arbitrary graph $G$ on $n+1$ nodes. Order the components
so that the maximal labels in the components are decreasing. Perform the
NFS on each component of~$G$ (see Section~\ref{s:comb}).
The result is a labeled forest on $n+1$ nodes, we denote it by $\Phi(G) = F$.
If $v$ has the maximal label in its component and there are $l$ nodes in
previous components, choose the coordinate of $v$ to be $x_l$. In other words,
$l$ is the position of the node in NFS. In particular, the coordinate of the node with label $n+1$ is $x_0$, which we set equal to $1$.
Every other node $v$ has a coordinate of the form $x_i/2^j - x_l$, where $i$ is its
position in NFS, $j$ is the number of cane paths in $F$ starting in $v$, and $l$ is the
maximal label in the component of $v$. Denote the coordinate of the node with label $k$ in a forest $F$ by $c(k,F)$.

\smallskip

Define $\alpha(F) = \sum_k j_k$, where the sum is over nodes that do not have
maximal labels in their components, and the coordinate of the node $k$ is $x_{i_k}/2^{j_k} - x_{l_k}$.
By Lemma \ref{alpha}, $j_k$ is the number of cane paths starting in the node, and $\alpha(F)$
is the number of cane paths in the forest $F$.

\smallskip

Now arrange the coordinates of the nodes $1,\ldots,n+1$ according to the labels. More precisely, define
$$\pS_F = \set{(x_1,\ldots,x_n) \colon 0 \leq c(1,F) \leq c(2,F) \leq \ldots \leq c(n+1,F) = 1}.
$$
See Example~\ref{exm4}.

\smallskip

The two definitions of $\pS_T$ for a tree coincide. Indeed, all the nodes except the one with label $n+1$ have
coordinates of the form $x_i/2^j - 1$, and adding $1$ to all the inequalities from the
new definition of $\pS_T$ gets the inequalities in the first definition.

\begin{thm} \label{thm4}
 For every labeled forest $F$ on $n+1$ nodes, the set $\pS_F$ is a simplex
 (but not in general an orthoscheme), and
 $$
 n!\ts\vol \pS_F \. = \. |\set{G \in \cG_{n+1} \, \,\mbox{\textup{s.t.}} \ \Phi(G) = F}|
 \. = \. 2^{\alpha(F)}\ts.
 $$
Furthermore, simplices $\pS_F$ triangulate the Gayley polytope $\pG_n$. In particular,
 $$
 n!\ts\vol \pG_n \. = \. |\set{G\in \cG_{n+1}}| \. = \. 2^{\binom{n+1}2}\ts.
 $$
\end{thm}

Although we already have a simple closed formula for the volume of Gayley polytopes,
this result is a stepping stone towards our studies of Tutte polytopes (see below).
The proof of the theorem is given in Section~\ref{proofs}, and follows the same
pattern as the proof of Theorem~\ref{thm2}.

\smallskip

By analogy with Cayley polytopes, let us show that Gayley polytope can
also be subdivided into a smaller number,
$\cat(n+1)$, of polytopes.  Given $\s P \subset \R^n$, define by
$$a \s P = \set{(ax_1,\ldots,ax_n) \colon (x_1,\ldots,x_n) \in \s P} \subset \R^n$$
the dilation of $\s P$ by $a \in \R$, and by
$$\cone(\s P) = \set{(x_0,x_1,\ldots,x_n) \colon 0 \leq x_0 \leq 1, (x_1,\ldots,x_n) \in x_0 \s P} \subset \R^{n+1}.$$
the cone with apex $0$ and base $\set 1 \times \s P$.

\smallskip

For an arbitrary graph $G$ on $n+1$ nodes, find the corresponding labeled forest $\Phi(G)$ and delete the labels to get a plane forest $\Psi(G)$ on $n+1$ nodes. For a plane forest $F$ on $n+1$ with components (plane trees) $T_1,T_2,T_3,\ldots$, define
$$\pD_F \. = \. \pD_{T_1} \times \cone(\pD_{T_2} \times \cone(\pD_{T_3} \times \cdots)).
$$

\begin{prop} \label{prop1}
 Take a plane forest $F$. For a node $w$ that is a root of its component, define coordinate $c(w,F) = x_l$, where $l$ is its position in NFS (equivalently, the components to the left have $l$ nodes total). For a node $v \neq w$ in the same component, define $c(v,F) = x_i/2^j - x_l$, where $i$ is its position in NFS and $j$ is the number of cane paths in $F$ starting in $v$. \.
 For each node with successors $v_1,\ldots,v_k$ (from left to right), take inequalities
 $$0 \leq c(v_1,F) \leq \ldots c(v_k,F) \leq c(w,F).$$
 Furthermore, if $w_1,\ldots,w_m$ are the roots of $F$ (from left to right), take inequalities
 $$0 \leq c(w_m,F) \leq \ldots \leq c(w_1,F) = 1.$$
 The resulting polytope is precisely $\pD_F$.
\end{prop}

See Example~\ref{exm5}.  We need this proposition for the following theorem, aimed
towards generalizations in the next sections.

\begin{thm} \label{thm5}
  For every plane forest $F$ on $n+1$ nodes, the set $\pD_F$ is a bounded polytope, and
 $$
 n!\ts\vol \pD_F \. = \. \left|\set{G \in \cG_{n+1} \, \mbox{\textup{s.t.}} \ \Psi(G) = F}\right|
 \. = \. \frac{\binom n{d_1,d_2,\ldots}} {\prod_{j=2}^m(a_j+\ldots+a_m)} \cdot
 2^{\binom{n+2-m}2 - \sum_{i=1}^{n+1-m} i d_i}\ts,
$$
where $(d_1,\ldots,d_{n+1-m})$ is the reduced degree sequence of~$F$.  Furthermore,
polytopes $\pD_F$ form a subdivision of the Gayley polytope~$\pG_n$.  In particular,
$$
n!\ts\vol \pG_n \. = \. \left|\set{G \in \cG_{n+1}}\right| \. = \. 2^{\binom{n+1}2}.
$$
\end{thm}

\medskip

\section{$t$-Cayley and $t$-Gayley polytopes} \label{t}

The constructions from the previous sections are easily adapted to weighted generalizations. Our presentation, the order and even shape of the results mimic the sections on Cayley and Gayley polytopes.  All proofs are moved to Section~\ref{proofs}, as before.

\smallskip

For $t \geq 0$, define the \emph{$t$-Cayley polytope} $\pC_n(t)$ and the \emph{$t$-Gayley polytope} $\pG_n(t)$ by replacing all $2$'s in the definition by $1 + t$. More precisely, define
$$\pC_n(t) = \set{(x_1,\ldots,x_n) \colon 1 \leq x_1 \leq 1 + t, 1 \leq x_i \leq (1 + t) x_{i-1} \mbox{ for } i = 2,\ldots,n}$$
and
$$\pG_n(t) = \set{(x_1,\ldots,x_n) \colon 0 \leq x_1 \leq 1 + t, 0 \leq x_i \leq (1 + t) x_{i-1} \mbox{ for } i = 2,\ldots,n}.$$

We can triangulate the polytopes $\pC_n(t)$ and $\pG_n(t)$ (or subdivide them into larger polytopes like in Sections~\ref{another} and~\ref{gayley}) in a very similar fashion as $\pC_n$ and $\pG_n$. For a labeled tree $T$ on $n+1$ nodes, attach a coordinate of the form $x_i/(1 + t)^j$ to each node $v$ of $T$, where $i$ is the position of $v$ in NFS, and $j$ is the number of cane paths starting in $v$. Arrange the coordinates of the nodes according to the labels. More precisely, define
$$\pS_T(t) = \set{(x_1,\ldots,x_n) \colon 1 \leq x_{i_1}/(1 + t)^{j_1} \leq x_{i_2}/(1 + t)^{j_2} \leq \ldots \leq x_{i_n}/(1 + t)^{j_n} \leq 1 + t},$$
where the coordinate of the node with label $k$ is $x_{i_k}/(1 + t)^{j_k}$. Note that the simplices $\pS_T(t)$ are also orthoschemes (see Example~\ref{exm6}). 

\begin{thm} \label{thm6}
 For every labeled tree $T$ on $n+1$ nodes, the set $\pS_T(t)$ is a simplex, and
 $$
 n!\ts\vol \pS_T(t) \, = \, \sum_{G\in \cC_{n+1}, \, \Phi(G) = T} t^{|E(G)|}
 \, = \, t^n (1 + t)^{\alpha(T)}\ts.
 $$
Furthermore, simplices $\pS_T(t)$ triangulate the $t$-Cayley polytope $\pC_n(t)$. In particular,
 $$n!\ts\vol \pC_n(t) \, = \, \sum_{G\in \cC_{n+1}} t^{|E(G)|}\ts.
 $$
 \end{thm}

A similar construction works for the other subdivision. As in the non-weighted case, erase all labels from the labeled tree $\Phi(G)$ to make it into a plane tree $\Psi(G)$. For each node $v$ with successors with coordinates $x_{i}/(1 + t)^j,x_{i+1}/(1 + t)^{j+1},\ldots,x_{i+k-1}/(1 + t)^{j+k-1}$, take inequalities
$$1 \leq x_{i+k-1}/(1 + t)^{j+k-1} \leq \ldots \leq x_{i+1}/(1 + t)^{j+1} \leq x_{i}/(1 + t)^j \leq 1 + t.$$
Denote the resulting polytope $\s D_T(t)$ (see Example~\ref{exm7}).

\begin{thm} \label{thm7}
 For every plane tree $T$ on $n+1$ nodes, the set $\pD_T(t)$ is a bounded polytope, and
 $$
 n!\ts\vol \pD_T(t) \. = \. \sum_{G\in \cG_{n+1}\ts, \ \Psi(G) = T} \ts t^{|E(G)|}
 \. = \. t^n (1 + t)^{\binom{n+1}2-\sum_{i=1}^{n+1} id_i} \cdot \binom n{d_1,d_2,\ldots},
 $$
 where $(d_1,\ldots,d_{n+1})$ is the degree sequence of $T$. Furthermore,
 polytopes $\pD_T(t)$ form a subdivision of the Cayley polytope $\pC_n(t)$. In particular,
 $$
 n!\ts\vol \pC_n(t) \. = \. \sum_{G\in \cG_{n+1}} \ts t^{|E(G)|}\ts.
 $$
 \end{thm}

Let us give a triangulation of the $t$-Gayley polytope. Take a labeled forest $F$ on $n+1$ nodes. If $v$ has the maximal label in its component and there are $i$ nodes in previous components, choose the coordinate of $v$ to be $t x_i$. In particular, the coordinate of the node with label $n+1$ is $t x_0 = t$. Every other node $v$ has a coordinate of the form $x_i/(1 + t)^j - x_l$, where $i$ is the position of $v$ in NFS, $j$ is the number of cane paths in $F$ starting in $v$, and $l$ is the maximal label in the component of $v$. Denote the coordinate of the node with label $k$ in a forest $F$ by $c(k,F;t)$.

\smallskip

Now arrange the coordinates of the nodes according to the labels. More precisely, define
$$\pS_F(t) = \set{(x_1,\ldots,x_n) \colon 0 \leq c(1,F;t) \leq c(2,F;t) \leq \ldots \leq c(n+1,F;t) = t}.$$
See Example~\ref{exm8}.

\begin{thm} \label{thm8}
 For every labeled forest $F$ on $n+1$ nodes, the set $\pS_F(t)$ is a simplex (but not in general an orthoscheme), and
 $$n!\ts\vol \pS_F(t) \. = \. \sum_{G\in \cG_{n+1}\ts, \ \Phi(G) = F} t^{|E(G)|} \. = \. t^{|E(F)|} (1 + t)^{\alpha(F)}.$$
Furthermore, simplices $\pS_F(t)$ triangulate the $t$-Gayley polytope $\pG_n(t)$. In particular,
 $$n!\ts\vol \pG_n(t) \. = \. \sum_{G\in \cG_{n+1}} t^{|E(G)|} \. = \. (1 + t)^{\binom {n+1} 2}.
 $$
 \end{thm}

We can also subdivide the $t$-Gayley polytope into a smaller number, $\cat(n+1)$, of polytopes. Recall that for an arbitrary graph $G$ on $n+1$ nodes, we have found the corresponding labeled forest $\Phi(G)$ and deleted the labels to get a plane forest $\Psi(G)$ on $n+1$ nodes. For a plane forest $F$ on $n+1$ nodes with components (plane trees) $T_1,\ldots,T_k$, define
$$\pD_F(t) = \pD_{T_1}(t) \times \cone(\pD_{T_2}(t) \times \cone(\pD_{T_3}(t) \times \cdots)).$$

\begin{prop} \label{prop2}
 Take a plane forest $F$. For a node $w$ that is a root of its component, define coordinate $c(w,F;t) = t x_l$, where $l$ is its position in NFS (equivalently, the components to the left have $l$ nodes total). For a node $v \neq w$ in the same component, define $c(v,F;t) = x_i/(1+t)^j - x_l$, where $i$ is its position in NFS and $j$ is the number of cane paths in $F$ starting in $v$. \.
 For each node with successors $v_1,\ldots,v_k$ (from left to right), take inequalities
 $$0 \leq c(v_1,F;t) \leq \ldots c(v_k,F;t) \leq c(w,F;t).$$
 Furthermore, if $w_1,\ldots,w_m$ are the roots of $F$ (from left to right), take inequalities
 $$0 \leq c(w_m,F;t) \leq \ldots \leq c(w_1,F;t) = t.$$
 The resulting polytope is precisely $\pD_F(t)$.
\end{prop}

See Example~\ref{exm9}.

\begin{thm} \label{thm9}
  For every plane forest $F$ on $n+1$ nodes, the set $\pD_F(t)$ is a bounded polytope, and
  $$n!\ts\vol \pD_F(t) \. = \!\!\!\!\! \sum_{G\in \cG_{n+1}, \Psi(G) = F} \!\!\!\!\!\! \!t^{|E(G)|}
  = \.
  \frac{\binom n{d_1,d_2,\ldots}} {\prod_{j=2}^m(a_j+\ldots+a_m)}
  t^{\sum_{i=1}^{n+1-m} d_i }(1 + t)^{\binom{n+2-m}2 - \sum_{i=1}^{n+1-m} i d_i}$$
 where $(d_1,\ldots,d_{n+1-m})$ is the reduced degree sequence of~$F$. Furthermore,
 polytopes $\pD_F(t)$ form a subdivision of the $t$-Gayley polytope $\pG_n(t)$.
 In particular,
  $$n!\ts\vol \pG_n(t) \. = \. \sum_{G\in \cG_{n+1}} t^{|E(G)|} \. = \. (1 + t)^{\binom {n+1} 2}.
 $$
 \end{thm}

\medskip

\section{The Tutte polytope}\label{s:tutte}

Recall that we defined the \emph{Tutte polytope} by inequalities
$$
 q x_i \leq q (1 + t)x_{i-1} - t(1-q)(1-x_{j-1}),
 $$
 where $1 \leq j \leq i \leq n$ and $x_0 = 1$. Here $0 < q \leq 1$ and $t > 0$. We have already established that it specializes to:
\begin{itemize}
 \item the Cayley polytope for $q = 0$, $t = 1$,
 \item the Gayley polytope for $q = 1$, $t = 1$,
 \item the $t$-Cayley polytope for $q = 0$,
 \item the $t$-Gayley polytope for $q = 1$.
\end{itemize}

In this section, we construct a triangulation and a subdivision of this polytope that prove Theorem \ref{t:main}. Recall that in the previous section, we were given a labeled forest $F$ and we attached a coordinate of the form $c(l,F;t) = t x_l$ to every root of the forest (where $x_0 = 1$), and $c(i,F;t) = x_i/(1+t)^j - x_l$ to every non-root node. Now the role of the former will be played by
$$c(l,F;q,t) = t(x_l-1+q)\ts,
$$
 and of the latter by
$$
 c(i,F;q,t) = \frac{qx_i - (1-q)(1-x_l)}{(1+t)^j} - (x_l - 1 + q)\ts.
$$
 Note that $c(i,F;1,t) = c(i,F;t)$ for all $i$.
Define
$$
\pS_F(q,t) = \set{(x_1,\ldots,x_n) \colon 0 \leq c(1,F;q,t) \leq c(2,F;q,t) \leq \ldots \leq c(n+1,F;q,t) = qt}\ts.
$$

\begin{thm} \label{thm10}
 For every labeled forest $F$ on $n+1$ nodes, the set $\pS_F(q,t)$ is a simplex, and
 $$n!\ts\vol \pS_F(q,t) \. = \.
 \sum_{G\in \cG_{n+1}\ts, \ \Phi(G) = F} q^{k(G) - 1} t^{|E(G)|} = q^{k(F) - 1} t^{|E(F)|} (1 + t)^{\alpha(F)}.$$
 Furthermore, simplices $\pS_F(q,t)$ triangulate the Tutte polytope $\pT_n(q,t)$.
 In particular,
 $$n!\ts\vol \pT_n(q,t) \. = \. \sum_{G\in \cG_{n+1}} q^{k(G)-1} t^{|E(G)|}.$$
 In other words,
 $$(\heartsuit) \ \ \
 n!\ts\vol \pT_n(q,t) \. = \. \rZ_{K_{n+1}}(q,t).$$
\end{thm}

This is a key result in this paper which implies Main Theorem (Theorem~\ref{thm2}).  The proof
is based on an extension of the previous results for $t$-Cayley and $t$-Gayley polytopes.  Although
the technical details are quite a bit trickier in this case, the structure of the proof follows
the same pattern as before.

\smallskip

For $q > 0$ and $\s P \subset \R^n$, define
$$
\cone_q(\s P) = \set{(x_0,x_1,\ldots,x_n) \colon 1-q \leq x_0 \leq 1,
q(x_1,\ldots,x_n) \in  (x_0-1+q)\s P + (1-q)(1-x_0)}.$$
the cone with apex $(1-q,\ldots,1-q)$ and base $\set 1 \times \s P$.

\smallskip

 For a plane forest $F$ on $n+1$ with components (plane trees) $T_1,T_2,T_3,\ldots$, define
$$D_F(q,t) = \s D_{T_1}(t) \times \cone_q(\s D_{T_2}(t) \times \cone_q(\s D_{T_3}(t) \times \cdots)).$$

\begin{prop} \label{prop3}
 Take a plane forest $F$. For a node $w$ that is a root of its component, define coordinate $c(w,F;q,t) = t (x_l - 1 + q)$, where $l$ is its position in NFS (equivalently, the components to the left have $l$ nodes total). For a node $v \neq w$ in the same component, define
 $$c(v,F;q,t) \. = \. \frac{qx_i - (1-q)(1-x_l)}{(1+t)^j} \. - \. (x_l - 1 + q)\.,
 $$
 where $i$ is its position in NFS and $j$ is the number of cane paths in $F$ starting in $v$. \.
 For each node with successors $v_1,\ldots,v_k$ (from left to right), take inequalities
 $$0 \leq c(v_1,F;q,t) \leq \ldots \leq c(v_k,F;q,t) \leq c(w,F;q,t)\ts.$$
 Furthermore, if $w_1,\ldots,w_m$ are the roots of $F$ (from left to right), take inequalities
 $$0 \leq c(w_m,F;q,t) \leq \ldots \leq c(w_1,F;q,t) = tq\ts.
 $$
 The resulting polytope is precisely $\pD_F(q,t)$.
\end{prop}

This proposition is used to prove the following result of independent interest,
a theorem which is in turn used to derive Theorem~\ref{thm10} in Section~\ref{proofs}.

\begin{thm} \label{thm12}
For every plane forest $F$ on $n+1$ nodes, the set $\pD_F(q,t)$ is a bounded polytope, and
$$
n!\ts\vol \pD_F(q,t) \. = \. \sum_{G\in \cG_{n+1}\ts, \ \Psi(G) = F} q^{k(G)-1} t^{|E(G)|} \. =
$$
$$= \. \frac{\binom n{d_1,d_2,\ldots}}{\prod_{j=2}^m(a_j+\ldots+a_m)} \, q^{k(F)-1} t^{\sum_{i=1}^{n+1-m} d_i }(1 + t)^{\binom{n+2-m}2 \ts -\ts
\sum_{i=1}^{n+1-m} \ts i \ts d_i}\.,
$$
where $(d_1,\ldots,d_{n+1-m})$ is the reduced degree sequence of $F$.  Furthermore,
polytopes $\pD_F(q,t)$ form a subdivision of the Tutte polytope $\pT_n(q,t)$. In particular,
$$
  n!\ts\vol \s T_n(t) \. = \. \sum_{G\in \cG_{n+1}} q^{k(G)-1} t^{|E(G)|}
  \. = \. \rZ_{K_{n+1}}(q,t)\ts.
$$
 \end{thm}

%



\medskip

\section{Vertices} \label{vertices}

The inequalities defining the Tutte polytope, as well as the simplices in the triangulation, are quite complicated compared to the ones for $t$-Cayley and $t$-Gayley polytopes. In this section, we see that the vertices of all the polytopes involved are very simple.

\smallskip

The following propositions give the vertices of the simplices $\pS_F(t)$ for $F$ a labeled forest, and of the $t$-Cayley polytope.

\begin{prop} \label{vertsx}
 Pick $t > 0$ and a labeled forest $F$ on $n+1$ nodes. The set of vertices of the simplex $\pS_F(t)$ is the set $V(F;t) = \set{v_p(F;t), 1 \leq p \leq n+1}$, where $v_p(F;t) = (x_1,\ldots,x_n)$ satisfies the following:
 \begin{enumerate}
  \item if the node $v$ is the $l$-th visited and its label $r$ is maximal in its component, then
  $$x_l = \left\{ \begin{array}{ccl} 1 & : & p \leq r \\ 0 & : & p > r \end{array} \right.$$
  \item if the node $v$ is the $i$-th visited and its label $k$ is not $r$, the maximal label in its component, then
  $$x_i = \left\{ \begin{array}{ccl} (1+t)^{j+1} & : & p \leq k \\ (1+t)^j & : & k < p \leq r \\ 0 & : & p > r \end{array}\right.$$
  where $j$ is the number of cane paths in $F$ starting in $v$.
 \end{enumerate}
\end{prop}

\begin{prop} \label{vertcayley}
 For $t > 0$, the set of vertices of $\pC_n(t)$ is the set
 $$
 V_n(t) = \bigl\{(x_1,\ldots,x_n) \colon \. x_1 \in \{1,1 + t\}, \, x_i \in \{1,(1 + t)\ts x_{i-1}\} \ \, \text{for} \ \. i = 2,\ldots,n\bigr\}\ts.
 $$
\end{prop}

Examples~\ref{exm11} and~\ref{exm10} illustrate these propositions.
For a labeled forest $F$, let $V(F;q,t)$ be the set of points that we get if we replace the (trailing) $0$'s in the coordinates of the points in $V(F;t)$ by $1-q$ (see Example~\ref{exm13}).

\begin{prop} \label{vertqsx}
 For a labeled forest $F$ and $t > 0$, $0 < q \leq 1$, $V(F;q,t)$ is the set of vertices of $\pS_F(q,t)$.
\end{prop}

Let $V_n(q,t)$ be the set $V_n(t)$ in which we replace the trailing $1$'s of each point by $1-q$ (see Example~\ref{exm12}).
We conclude with the main result of this section:

\begin{thm} \label{verttutte}
 For $t > 0$ and $0 \leq q < 1$, $V_n(q,t)$ is the set of vertices of $\s T_n(q,t)$. In particular, the Tutte polytope $\s T_n(q,t)$ has $2^n$ vertices.
\end{thm}

\medskip

\section{Application: a recursive formula}\label{s:app}

The results proved in this paper yield an interesting recursive formulas for the generating function for (or the number of) labeled connected graphs. Let
$$F_n(t) \. = \. \sum t^{|E(G)|} \. = \. t^{n-1} \ts \Inv_n(1+t)\ts,$$
where the sum is over labeled connected graphs on $n$ nodes.

\begin{thm} \label{rec2}
 Define polynomials $r_n(t)$, $n \geq 0$, by
 $$r_0(t) = 1, \qquad r_n(t) = - \sum_{j = 1}^{n} \binom n{j} (1 + t)^{\binom{j}2} r_{n-j}(t).$$
 Then
 $$F_{n+1}(t) = \sum_{j=0}^n \binom n j (1 + t)^{\binom{j+1}2} r_{n-j}(t).$$
\end{thm}
\begin{proof}
 Define
$$I_n(x,t) = n! \int_1^{(1+t)x} dx_1\int_1^{(1+t)x_1}dx_2 \int_1^{(1+t)x_2} dx_3 \cdots \int_1^{(1+t)x_{n-1}} dx_n.$$
Clearly, we have
$$I_n(x,t) = n \int_1^{(1+t)x} I_{n-1}(x_1,t) dx_1.$$
We claim that
$$I_n(x,t) = \sum_{j = 0}^n \binom n j (1+t)^{\binom{j+1}2} r_{n-j}(t) x^j,$$
where the polynomials $r_j(t)$ are defined in Theorem \ref{rec2}. Indeed, the statement is obviously true for $n = 0$, and by induction,
$$I_n(x,t) = n \int_1^{(1+t)x} \left( \sum_{j = 0}^{n-1} \binom {n-1} j (1+t)^{\binom{j+1}2} r_{n-1-j}(t) x_1^j\right) dx_1 = $$
$$= n \sum_{j = 0}^{n-1} \binom {n-1} j (1+t)^{\binom{j+1}2} r_{n-1-j}(t) \left( \frac{(1+t)^{j+1}}{j+1} x^{j+1} - \frac 1{j+1}\right) =$$
$$= \sum_{j = 1}^n \binom n j(1+t)^{k+1} r_{n-j}(t)x^j - \sum_{j = 0}^{n-1} \binom n{j+1} (1+t)^{\binom{j+1} 2} r_{n-1-j}(t),$$
which by definition of $r_n(t)$ equals
$$ \sum_{j = 0}^n \binom n j (1+t)^{\binom{j+1}2} r_{n-j}(t) x^j. $$
Theorem \ref{rec2} follows by plugging in $x = 1$ into the equation.
\end{proof}

\medskip

\section{Examples} \label{examples}

\begin{exa} \label{exm1} {\rm \small
Let $G$ be the graph on the left-hand side of Figure~\ref{fig5}. The neighbors-first search starts in node $12$ and visits the other nodes in order $11,10,6,8,7,9,3,1,4,2,5$. The resulting search tree $\Phi(G)$ is on the right-hand side of Figure~\ref{fig5}. The edges of $G$ that are not in $\Phi(G)$ are dashed.

\begin{figure}[ht!]
 \begin{center}
   \includegraphics[height=4cm]{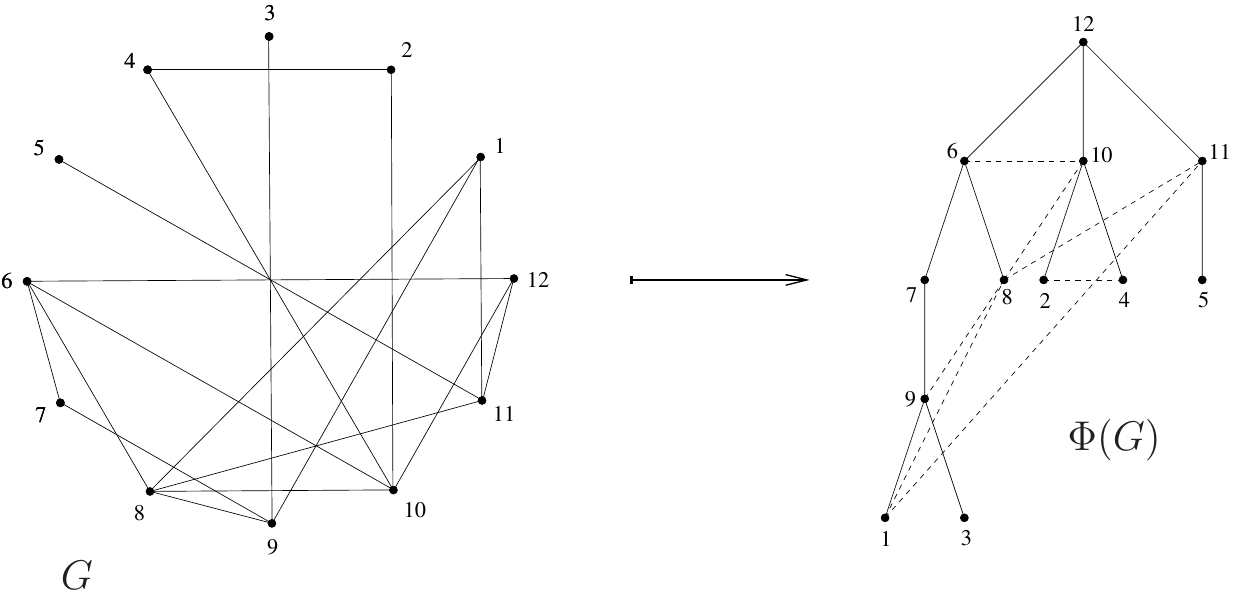}
   \caption{A connected graph and its NFS tree.}
   \label{fig5}
 \end{center}
\end{figure}

}\end{exa}

\begin{exa} \label{exm2} {\rm \small
For the tree $T$ drawn on the right-hand side of Figure~\ref{fig5}, the
coordinates of the nodes with labels $1,\ldots,11$ are shown in
Figure~\ref{fig7}. We have $\alpha(T) = 21$.

 \begin{figure}[ht!]
 \begin{center}
   \includegraphics[height=4cm]{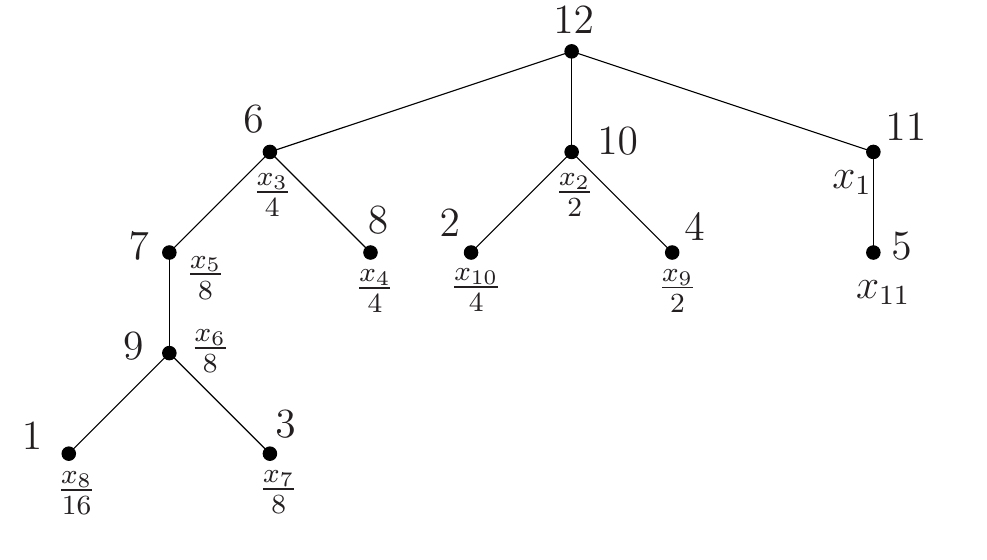}
   \caption{Coordinates in a tree.}
   \label{fig7}
 \end{center}
\end{figure}
The corresponding set $\pS_T$ is the set of points $(x_1,\ldots,x_{11})$ satisfying
$$
1 \leq \frac{x_8}{16} \leq \frac{x_{10}}4  \leq \frac{x_7}8  \leq \frac{x_9}2
\leq x_{11}  \leq  \frac{x_3}4 \leq \frac{x_5}8  \leq \frac{x_4}4  \leq \frac{x_6}8
\leq \frac{x_2}2  \leq x_1 \leq 2\ts.
$$
} \end{exa}

\begin{exa} \label{exm3} {\rm \small
 For the graph $G$ from the left-hand side of Figure~\ref{fig5}, $\s D_{\Psi(G)}(t)$ is


$$\left\{\begin{array}{llll}
             (x_1,\ldots,x_{11}) \colon & 1 \leq x_ 1 \leq 2, & 2 \leq x_ 2 \leq 2\ts x_ 1, & 4 \leq x_3 \leq 2\ts x_ 2,\\
              4 \leq x_ 4 \leq 8, & 8 \leq x_ 5 \leq 2\ts x_4, & 8 \leq x_ 6 \leq 16, & 8 \leq x_ 7 \leq 16,\\
	    16 \leq x_ 8 \leq 2\ts x_ 7, & 2 \leq x_ 9 \leq 4, &4 \leq x_{10} \leq 2\ts x_ 9,&1 \leq x_{11} \leq 2
            \end{array}\right\}.$$


}\end{exa}

\begin{exa} \label{exm4} {\rm \small
 Take $G$ to be the (disconnected) graph on the left-hand side of
 Figure~\ref{fig10}. The search forest $F$ of the NFS is given on the right.
\begin{figure}[ht!]
 \begin{center}
   \includegraphics[height=4cm]{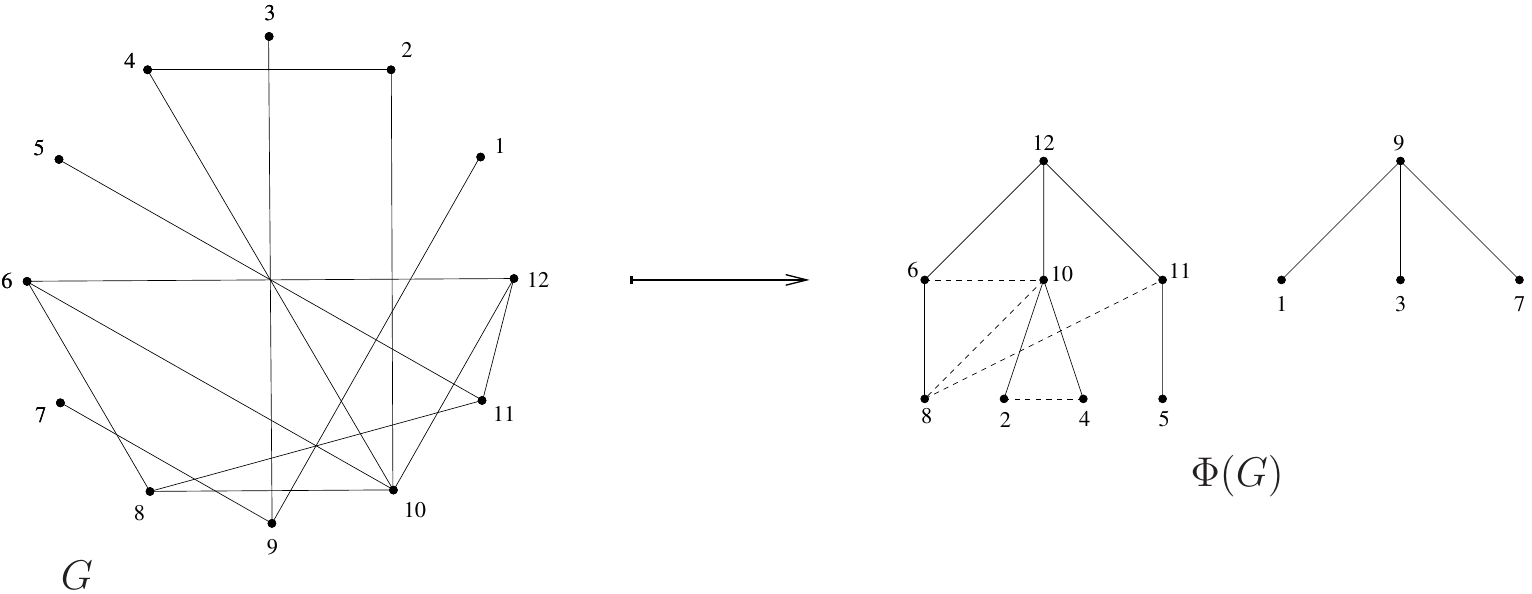}
   \caption{A disconnected graph and its NFS forest.}
   \label{fig10}
 \end{center}
\end{figure}
 Figure~\ref{fig11} illustrates the coordinates we attach to the nodes of the forest $F$. The corresponding set $\pS_F$ is the set of points $(x_1,\ldots,x_{11})$ satisfying
$$0 \leq \frac{x_{11}}4 - x_8 \leq  \frac{x_{6}}4 - 1 \leq  \frac{x_{10}}2 - x_8 \leq  \frac{x_{5}}2 - 1 \leq  x_{7} - 1 \leq $$
$$ \leq  \frac{x_{3}}4 - 1  \leq x_{9} - x_8 \leq  \frac{x_{4}}4 - 1 \leq x_8  \leq \frac{x_{2}}2 - 1 \leq  x_{1} - 1 \leq 1.$$

 \begin{figure}[ht!]
 \begin{center}
   \includegraphics[height=4.1cm]{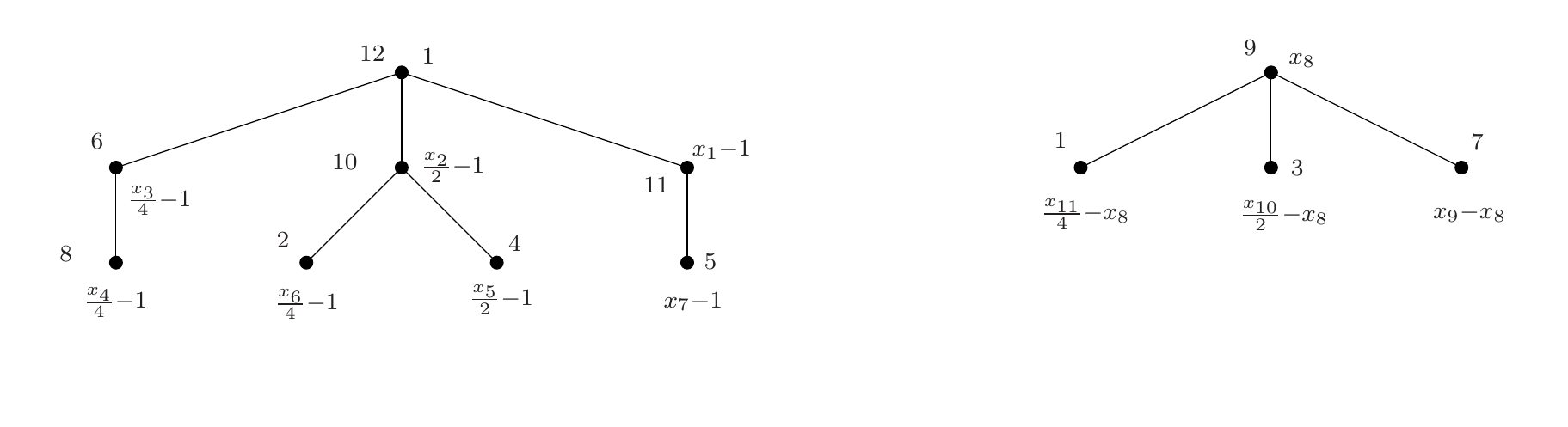}
   \caption{Coordinates in a forest.}
   \label{fig11}
 \end{center}
\end{figure}

}\end{exa}

\begin{exa} \label{exm5} {\rm \small
 For the right component $T_2$ of the forest on the right-hand side of Figure~\ref{fig10},
$$
\pD_{T_2} \. = \. \set{(x_1,x_2,x_3) \colon 1 \leq x_1 \leq 2\ts, \. 2 \leq x_2 \leq 2\ts x_1\ts, \. 4 \leq x_3 \leq 2\ts x_2},
$$
so
$$
\cone(\pD_{T_2}) \. = \.
\left\{(x_0,x_1,x_2,x_3) \colon 0 \leq x_0 \leq 1, \. 1 \leq \frac{x_1}{x_0} \leq 2, \. 2 \leq \frac{x_2}{x_0} \leq 2\ts \frac{x_1}{x_0}, \. 4 \leq \frac{x_3}{x_0} \leq 2\ts \frac{x_2}{x_0}\right\}.
$$
 For the graph $G$ from the left-hand side of Figure~\ref{fig10}, $\s D_{\Psi(G)}$ is
$$\left\{\begin{array}{llll}
             (x_1,\ldots,x_{11}) \colon & 1 \leq x_ 1 \leq 2, & 2 \leq x_ 2 \leq 2\ts x_ 1, & 4 \leq x_3 \leq 2\ts x_ 2,\\
              4 \leq x_ 4 \leq 8, & 2 \leq x_ 5 \leq 4, & 4 \leq x_ 6 \leq 2\ts x_5, & 1 \leq x_ 7 \leq 2,\\
	    0 \leq x_ 8 \leq 1, & x_8 \leq x_ 9 \leq 2\ts x_8, &2\ts x_8 \leq x_{10} \leq 2\ts x_ 9,&4\ts x_8 \leq x_{11} \leq 2\ts x_{10}
            \end{array}\right\}.$$

}\end{exa}

\begin{exa} \label{exm6} {\rm \small
 For the graph $G$ on the left-hand side of Figure~\ref{fig5}, $S_{\Phi(G)}(t)$ is the set of points $(x_1,\ldots,x_{11})$ satisfying
 $$1 \leq \frac{x_8}{(1 + t)^4} \leq \frac{x_{10}}{(1 + t)^2}  \leq \frac{x_7}{(1 + t)^3}  \leq \frac{x_9}{1 + t}  \leq x_{11}  \leq  \frac{x_3}{(1 + t)^2} \leq $$
$$\leq \frac{x_5}{(1 + t)^3}  \leq \frac{x_4}{(1 + t)^2}  \leq \frac{x_6}{(1 + t)^3}  \leq \frac{x_2}{1 + t}  \leq x_1 \leq 1 + t.$$
}\end{exa}

\begin{exa} \label{exm7} {\rm \small  For the graph $G$ from the left-hand side of Figure~\ref{fig5}, $\s D_{\Psi(G)}(t)$ is

$$\left\{\begin{array}{lll}
             (x_1,\ldots,x_{11}) \colon & 1 \leq x_ 1 \leq 1+t, & 1+t \leq x_ 2 \leq (1+t)\ts x_ 1, \\
	  (1+t)^2 \leq x_3 \leq (1+t)\ts x_ 2, & (1+t)^2 \leq x_4 \leq (1+t)^3, & (1+t)^3 \leq x_ 5 \leq (1+t)\ts x_4, \\
           (1+t)^3 \leq x_ 6 \leq (1+t)^4, & (1+t)^3 \leq x_ 7 \leq (1+t)^4, & (1+t)^4 \leq x_8 \leq (1+t)\ts x_ 7, \\
            1+t \leq x_ 9 \leq (1+t)^2, &(1+t)^2 \leq x_{10} \leq (1+t)\ts x_ 9,&1 \leq x_{11} \leq 1+t
            \end{array}\right\}.$$

}\end{exa}

\begin{exa} \label{exm8} {\rm \small
  Take $G$ to be the graph on the left-hand side of Figure~\ref{fig10}. Figure~\ref{fig12} illustrates the coordinates we attach to the nodes of the corresponding forest $F$. The corresponding set $\pS_F(t)$ is the set of points $(x_1,\ldots,x_{11})$ satisfying
$$0 \leq \frac{x_{11}}{(1 + t)^2} - x_8 \leq  \frac{x_{6}}{(1 + t)^2} - 1 \leq  \frac{x_{10}}{1 + t} - x_8 \leq  \frac{x_{5}}{1 + t} - 1 \leq  x_{7} - 1 \leq $$
$$ \leq  \frac{x_{3}}{(1 + t)^2} - 1  \leq x_{9} - x_8 \leq  \frac{x_{4}}{(1 + t)^2} - 1 \leq t x_8  \leq \frac{x_{2}}{1 + t} - 1 \leq  x_{1} - 1 \leq t.$$

 \begin{figure}[ht!]
 \begin{center}
   \includegraphics[height=4.0cm]{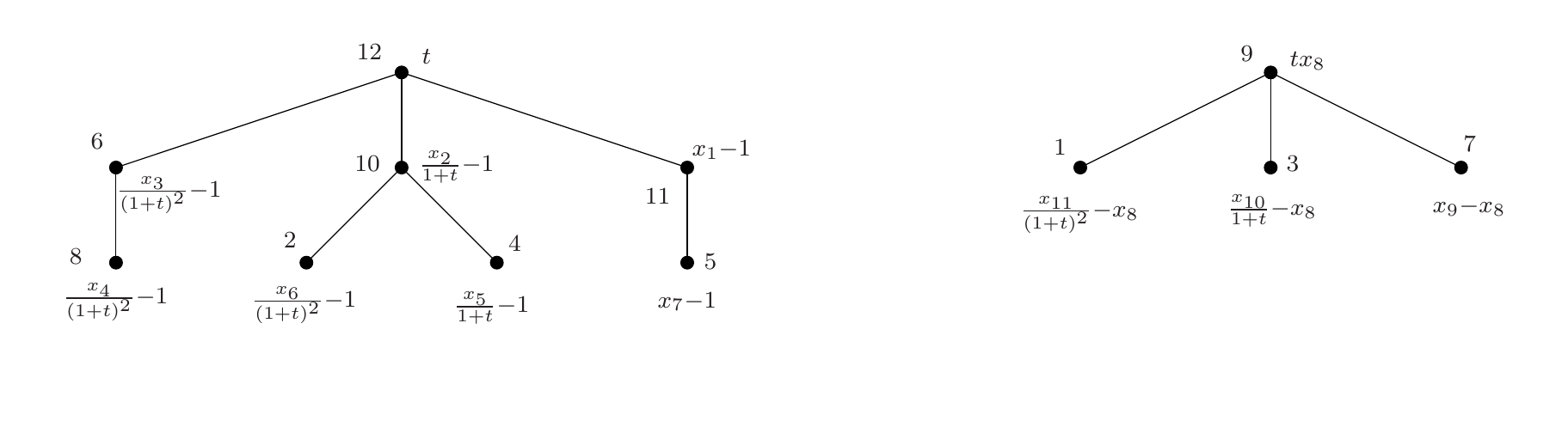}
   \caption{Coordinates in a forest (the $t$-Gayley case).}
   \label{fig12}
 \end{center}
\end{figure}

}\end{exa}

\begin{exa} \label{exm9} {\rm \small For the graph $G$ from the left-hand side of Figure~\ref{fig10}, $\s D_{\Psi(G)}(t)$ is

$$\left\{\begin{array}{lll}
             (x_1,\ldots,x_{11}) \colon & 1 \leq x_ 1 \leq 1+t, & 1+t \leq x_ 2 \leq (1+t)\ts x_ 1, \\
	  (1+t)^2 \leq x_3 \leq (1+t)\ts x_ 2, & (1+t)^2 \leq x_4 \leq (1+t)^3, & 1+t \leq x_ 5 \leq (1+t)^2, \\
           (1+t)^2 \leq x_ 6 \leq (1+t)\ts x_5, & 1 \leq x_ 7 \leq 1+t, & 0 \leq x_8 \leq 1, \\
            x_8 \leq x_ 9 \leq (1+t)\ts x_8, &(1+t)x_8 \leq x_{10} \leq (1+t)\ts x_ 9,&(1+t)^2\ts x_8 \leq x_{11} \leq (1+t)\ts x_{10}
            \end{array}\right\}.$$

}\end{exa}

\smallskip

\begin{exa} \label{exm11} {\rm \small
\setcounter{MaxMatrixCols}{11}
 The coordinates of the vertices of $\pS_F(t)$ for the forest from Figure~\ref{fig12} are given by lines in the following table:
 {\footnotesize
 $$
\begin{matrix}
1 + t & (1 + t)^2 & (1 + t)^3 & (1 + t)^3 & (1 + t)^2 & (1 + t)^3 & 1 + t & 1 & 1 + t & (1 + t)^2 & (1 + t)^3 \\
1 + t & (1 + t)^2 & (1 + t)^3 & (1 + t)^3 & (1 + t)^2 & (1 + t)^3 & 1 + t & 1 & 1 + t & (1 + t)^2 & (1 + t)^2 \\
1 + t & (1 + t)^2 & (1 + t)^3 & (1 + t)^3 & (1 + t)^2 & (1 + t)^2 & 1 + t & 1 & 1 + t & (1 + t)^2 & (1 + t)^2 \\
1 + t & (1 + t)^2 & (1 + t)^3 & (1 + t)^3 & (1 + t)^2 & (1 + t)^2 & 1 + t & 1 & 1 + t & 1 + t & (1 + t)^2 \\
1 + t & (1 + t)^2 & (1 + t)^3 & (1 + t)^3 & 1 + t & (1 + t)^2 & 1 + t & 1 & 1 + t & 1 + t & (1 + t)^2 \\
1 + t & (1 + t)^2 & (1 + t)^3 & (1 + t)^3 & 1 + t & (1 + t)^2 & 1 & 1 & 1 + t & 1 + t & (1 + t)^2 \\
1 + t & (1 + t)^2 & (1 + t)^2 & (1 + t)^3 & 1 + t & (1 + t)^2 & 1 & 1 & 1 + t & 1 + t & (1 + t)^2 \\
1 + t & (1 + t)^2 & (1 + t)^2 & (1 + t)^3 & 1 + t & (1 + t)^2 & 1 & 1 & 1 & 1 + t & (1 + t)^2 \\
1 + t & (1 + t)^2 & (1 + t)^2 & (1 + t)^2 & 1 + t & (1 + t)^2 & 1 & 1 & 1 & 1 + t & (1 + t)^2 \\
1 + t & (1 + t)^2 & (1 + t)^2 & (1 + t)^2 & 1 + t & (1 + t)^2 & 1 & 0 & 0 & 0 & 0 \\
1 + t & 1 + t & (1 + t)^2 & (1 + t)^2 & 1 + t & (1 + t)^2 & 1 & 0 & 0 & 0 & 0 \\
1 & 1 + t & (1 + t)^2 & (1 + t)^2 & 1 + t & (1 + t)^2 & 1 & 0 & 0 & 0 & 0
   \end{matrix}$$
}
}\end{exa}

\smallskip

\begin{exa} \label{exm10} {\rm \small The coordinates of the vertices of $\pC_3(t)$ are given by lines in the following table:

$$\begin{matrix}
1 + t & (1 + t)^2 & (1 + t)^3 \\
1 + t & (1 + t)^2 & 1 \\
1 + t & 1 & 1 + t \\
1 + t & 1 & 1 \\
1 & 1 + t & (1 + t)^2 \\
1 & 1 + t & 1 \\
1 & 1 & 1 + t \\
1 & 1 & 1
   \end{matrix}
$$
}\end{exa}

\begin{exa} \label{exm13} {\rm \small
  The coordinates of the vertices of $\pS_F(q,t)$ for the forest from Figure~\ref{fig12} are given by lines in the following table:

 $$
\begin{matrix}
1 + t & \!\!\: (1 + t)^2 & \!\!\: (1 + t)^3 & \!\!\: (1 + t)^3 & \!\!\: (1 + t)^2 & \!\!\: (1 + t)^3 & \!\!\: 1 + t & \!\!\: 1 & \!\!\: 1 + t & \!\!\: (1 + t)^2 & \!\!\: (1 + t)^3 \\
1 + t & \!\!\: (1 + t)^2 & \!\!\: (1 + t)^3 & \!\!\: (1 + t)^3 & \!\!\: (1 + t)^2 & \!\!\: (1 + t)^3 & \!\!\: 1 + t & \!\!\: 1 & \!\!\: 1 + t & \!\!\: (1 + t)^2 & \!\!\: (1 + t)^2 \\
1 + t & \!\!\: (1 + t)^2 & \!\!\: (1 + t)^3 & \!\!\: (1 + t)^3 & \!\!\: (1 + t)^2 & \!\!\: (1 + t)^2 & \!\!\: 1 + t & \!\!\: 1 & \!\!\: 1 + t & \!\!\: (1 + t)^2 & \!\!\: (1 + t)^2 \\
1 + t & \!\!\: (1 + t)^2 & \!\!\: (1 + t)^3 & \!\!\: (1 + t)^3 & \!\!\: (1 + t)^2 & \!\!\: (1 + t)^2 & \!\!\: 1 + t & \!\!\: 1 & \!\!\: 1 + t & \!\!\: 1 + t & \!\!\: (1 + t)^2 \\
1 + t & \!\!\: (1 + t)^2 & \!\!\: (1 + t)^3 & \!\!\: (1 + t)^3 & \!\!\: 1 + t & \!\!\: (1 + t)^2 & \!\!\: 1 + t & \!\!\: 1 & \!\!\: 1 + t & \!\!\: 1 + t & \!\!\: (1 + t)^2 \\
1 + t & \!\!\: (1 + t)^2 & \!\!\: (1 + t)^3 & \!\!\: (1 + t)^3 & \!\!\: 1 + t & \!\!\: (1 + t)^2 & \!\!\: 1 & \!\!\: 1 & \!\!\: 1 + t & \!\!\: 1 + t & \!\!\: (1 + t)^2 \\
1 + t & \!\!\: (1 + t)^2 & \!\!\: (1 + t)^2 & \!\!\: (1 + t)^3 & \!\!\: 1 + t & \!\!\: (1 + t)^2 & \!\!\: 1 & \!\!\: 1 & \!\!\: 1 + t & \!\!\: 1 + t & \!\!\: (1 + t)^2 \\
1 + t & \!\!\: (1 + t)^2 & \!\!\: (1 + t)^2 & \!\!\: (1 + t)^3 & \!\!\: 1 + t & \!\!\: (1 + t)^2 & \!\!\: 1 & \!\!\: 1 & \!\!\: 1 & \!\!\: 1 + t & \!\!\: (1 + t)^2 \\
1 + t & \!\!\: (1 + t)^2 & \!\!\: (1 + t)^2 & \!\!\: (1 + t)^2 & \!\!\: 1 + t & \!\!\: (1 + t)^2 & \!\!\: 1 & \!\!\: 1 & \!\!\: 1 & \!\!\: 1 + t & \!\!\: (1 + t)^2 \\
1 + t & \!\!\: (1 + t)^2 & \!\!\: (1 + t)^2 & \!\!\: (1 + t)^2 & \!\!\: 1 + t & \!\!\: (1 + t)^2 & \!\!\: 1 & \!\!\: 1-q & \!\!\: 1-q & \!\!\: 1-q & \!\!\: 1-q \\
1 + t & \!\!\: 1 + t & \!\!\: (1 + t)^2 & \!\!\: (1 + t)^2 & \!\!\: 1 + t & \!\!\: (1 + t)^2 & \!\!\: 1 & \!\!\: 1-q & \!\!\: 1-q & \!\!\: 1-q & \!\!\: 1-q \\
1 & \!\!\: 1 + t & \!\!\: (1 + t)^2 & \!\!\: (1 + t)^2 & \!\!\: 1 + t & \!\!\: (1 + t)^2 & \!\!\: 1 & \!\!\: 1-q & \!\!\: 1-q & \!\!\: 1-q & \!\!\: 1-q
   \end{matrix}
$$
}\end{exa}

\begin{exa} \label{exm12} {\rm \small
  The coordinates of the vertices of $\s T_3(t)$ are given by lines in the following table:
 $$\begin{matrix}
  1 + t & (1 + t)^2 & (1 + t)^3 \\
 1 + t & (1 + t)^2 & 1-q \\
 1 + t & 1 & 1 + t \\
 1 + t & 1-q & 1-q \\
 1 & 1 + t & (1 + t)^2 \\
 1 & 1 + t & 1-q \\
1 & 1 & 1 + t \\
 1-q & 1-q & 1-q \\
\end{matrix}$$
}\end{exa}

\medskip

\section{Proofs} \label{proofs}

The proofs proceed as follows. First, we prove Theorems~\ref{thm6} and~\ref{thm7}, which give a triangulation and a subdivision of the $t$-Cayley polytope. If we plug in $t = 1$, we get Theorems~\ref{thm2} and~\ref{thm3}. A relatively simple extension of the proof yields Theorems~\ref{thm8} and~\ref{thm9} (about the $t$-Gayley polytope) and hence (for $t = 1$) also Theorems~\ref{thm4} and~\ref{thm5}. The proofs for subdivisions of the Tutte polytope (Theorems~\ref{thm10} and~\ref{thm12}) are similar and we provide a detailed proof for only some of them. The statements from Section~\ref{vertices} are relatively straightforward.


\subsection{Subdivisions of the $t$-Cayley polytope}

\begin{proof}[Proof of Lemma~\ref{alpha}]
 We prove the lemma by induction on the rank (distance from the root) of the node. The root has coordinate $x_0/2^0$, and there are obviously no cane paths from it. If the coordinate of a node $v$ is $x_i/2^j$ with $j$ the number of cane paths starting in $v$, then by construction its children $v_k,\ldots,v_1$  have coordinates $x_{i'}/2^j,\ldots,x_{i'+k-1}/2^{j+k-1}$. A cane path starting in $v_l$ goes either up to $v$ and then down to $v_{l'}$ for $l' > l$, or it goes up to $v$ and then coincides with a cane path starting in $v$. In other words, there are $j + k - l$ cane paths starting in $v_l$, and the coordinate is indeed $x_{i'+k-l}/2^{j+k-l}$. This finishes the proof.
\end{proof}

Take a labeled tree $T$ and the corresponding $\pS_T$ defined by
$$1 \leq x_{i_1}/(1 + t)^{j_{i_1}} \leq x_{i_2}/(1 + t)^{j_{i_2}} \leq \ldots \leq x_{i_n}/(1 + t)^{j_{i_n}} \leq 1 + t.$$
Define the transformation
$$A \colon x_{i} \mapsto (1+t)^{j_i}(tx_{i} + 1).$$
 Then $A(\pS_T)$ is defined by
$$0 \leq x_{i_1} \leq x_{i_2} \leq \ldots \leq x_{i_n} \leq 1$$
and is hence a simplex with volume $1/n!$. Since $A$ is the composition of a linear transformation with determinant $t^n(1+t)^{\alpha(T)}$ and a translation, $\pS_T$ is a simplex with volume $t^n(1+t)^{\alpha(T)}/n!$.

\smallskip

Now, let us compute the generating function
$$\sum_{\Phi(G) = T} t^{|E(G)|},$$
where the sum runs over all labeled connected graphs that map to $T$. Every such graph has all the $n$ edges of $T$. Call an edge $e \notin E(T)$ a \emph{cane edge of $T$} if there exists a cane path from one of the nodes to the other.

\begin{lemma}
 For a labeled connected graph $G$ we have $\Phi(G) = T$ if and only if $E(G) = E(T) \cup C$, where $C$ is a subset of the set of cane edges of $T$.
\end{lemma}
\begin{proof}
 If $\Phi(G) = T$, then clearly $E(T) \subseteq E(G)$. Assume that there is an edge $e \in E(G) \setminus E(T)$ that is not a cane edge of $T$. Write $e = uv$, where $u$ is weakly to the left of $v$ (i.e.\ either $u$ is a descendant of $v$, or the unique path from $u$ to $v$ in $T$ goes up and then down right). Since $e$ is not a cane edge of $T$, the path in $T$ from $u$ to $v$ first has $k$, $k \geq 0$, up steps and then $l$, $l \geq 2$ down steps. But then when the NFS on $G$ visits $u$, $v$ is a previously unvisited neighbor of $v$, so $e$ is in the search tree, and $\Phi(G) \neq T$, which is a contradiction.

 For the other direction, assume that $E(G) = E(T) \cup C$, where $C$ is a set of cane edges of $T$. The neighbors of the node with label $n+1$ are the same in $G$ and $T$, so the beginning of the NFS is the same on $G$ and $T$. By induction, assume that the NFS visits the same nodes in the same order up to the node $v$. The edges from $v$ in $G$ are the same as in $T$, and possibly some cane edges of $T$. But all the cane edges are connected to previously visited nodes: these nodes are children of an ancestor of $v$ in $T$, or their are descendants of a left neighbor of $v$. In other words, no cane edge enters the search tree. That means that $\Phi(G) = T$.
\end{proof}

Since there are $\alpha(T)$ cane paths by Lemma~\ref{alpha}, this implies that
$$\sum_{\Phi(G) = T} t^{|E(G)|} = t^n (1+t)^{\alpha(t)},$$
which finishes the proof of the first part of Theorem~\ref{thm6}.



\begin{proof}[Proof of the first part of Theorem~\ref{thm7}]
Note that $\alpha(T)$ is the same for all labeled trees with the same underlying plane tree. Recall that $\pD_T$ for $T$ a plane tree is defined by determining the order of the coordinates of only the $d_i$ nodes with the same parent. There are clearly $\binom{n}{d_1,d_2,\ldots}$ ways to extend such orderings to an ordering of the coordinates of all nodes. In other words, $\pD_T$ is composed of $\binom{n}{d_1,d_2,\ldots}$ simplices with volume $t^n(1+t)^{\alpha(T)}$. There are also $\binom{n}{d_1,d_2,\ldots}$ ways to label a plane tree so that the labels of the nodes with the same parent are increasing from left to right. This proves that
$$n!\ts\vol \pD_T = \sum_{\Psi(G) = T} t^{|E(G)|} = t^n(1+t)^{\alpha(T)}\binom{n}{d_1,d_2,\ldots}.$$
It remains to express $\alpha(T)$ in terms of the degree sequence. Assume that the plane tree $T$ has a root with degree $k$ and subtrees $T_1,\ldots,T_k$ with $a_1,\ldots,a_k$ nodes. The number of cane paths in $T$ is equal to the number of cane paths that pass through the root, plus the number of cane paths in the trees $T_1,\ldots,T_k$. For a node in $T_j$, there are $k-j$ cane paths that start in that node and pass through the root of the tree. By induction, we have
$$\alpha(T) = \sum_{j=1}^k (k-j)a_j + \sum_{j=1}^k \alpha(T_j) = \sum_{j=1}^k (k-j)a_j + \sum_{j=1}^k \left( \binom{a_j}2 - \sum_{i=1}^{a_j} i d_{a_1+\ldots+a_{j-1}+i+1}\right) = $$
$$ =  \sum_{j=1}^k (k-j)a_j + \sum_{j=1}^k \left( \binom{a_j}2 - \sum_{i=a_1+\ldots+a_{j-1}+2}^{a_1+\ldots+a_{j}+1} (i - a_1-\ldots-a_{j-1}-1) d_i\right) =$$
$$ = \sum_{j=1}^k \left( (k-j)a_j + \binom{a_j} 2 + (a_1+\ldots +a_{j-1}+1)(a_j-1)\right) - \sum_{i=2}^{n+1} i d_i,$$
where we used the fact that
$$\sum_{i=a_1+\ldots+a_{j-1}+2}^{a_1+\ldots+a_{j}+1} d_i = a_j - 1.$$
It is easy to see that
$$
\sum_{j=1}^k \left( (k-j)a_j + \binom{a_j} 2 + (a_1+\ldots +a_{j-1}+1)(a_j-1)\right)
\. = \. \binom{a_1+\ldots+a_k+1}2 - k,
$$
which implies
$$\alpha(T) \. = \. \binom{n+1}2 - \sum_{i=1}^{n+1} i d_i
$$
and finishes the proof.
\end{proof}

\begin{lemma} \label{lemma1}
 For a labeled $($respectively, plane$)$ tree $T$ on $n+1$ vertices, $\s S_T(t) \subseteq \s C_n(t)$ $($respectively, $\s D_T(t) \subseteq \s C_n(t)$ \ts$)$.
\end{lemma}
\begin{proof}
 We only prove the statement for a labeled tree $T$ as the proof for a plane tree is almost identical. By construction, we have $x_i/(1+t)^j \geq 1$ for each $i$ and some $j \geq 0$, so $x_i \geq (1+t)^j \geq 1$. Also by construction, $1 \leq x_1 \leq 1+t$. Assume that $v$ is the $(i-1)$-st visited node and $v'$ the $i$-th visited node, where $i \geq 2$, and that their coordinates are $x_{i-1}/(1+t)^j$ and $x_{i}/(1+t)^{j'}$. We have several possibilities:
 \begin{itemize}
  \item $v'$ and $v$ have the same parent and $v'$ has a smaller label; in this case $j' = j+1$, $x_{i}/(1+t)^{j+1} \leq x_{i-1}/(1+t)^j$ and $x_i \leq (1+t)x_{i-1}$;
  \item $v'$ is the child of $v$ with the largest label; in this case, $j' = j$, so both $1 \leq x_{i-1}/(1+t)^{j} \leq 1+t$ and $1 \leq x_{i}/(1+t)^{j} \leq 1+t$ hold; that means that $x_i \leq  (1+t)^{j+1} \leq (1+t)x_{i-1}$;
  \item the unique path from $v$ to $v'$ goes up at least once, then down right, and then down to the child with the largest label; in this case, every cane path starting in $v'$ and ending in $w$ has a corresponding cane path starting in $v$ and ending in $w$, so $j' \leq j$, and we have both $1 \leq x_{i-1}/(1+t)^{j} \leq 1+t$ and $1 \leq x_{i}/(1+t)^{j'} \leq 1+t$, so $x_i \leq (1+t)^{j'+1} \leq (1+t)^{j+1} \leq (1+t)x_{i-1}$.
 \end{itemize}
This finishes the proof.
\end{proof}

The $t$-Cayley polytope $\pC_n(t)$ consists of all points $(x_1,\ldots,x_n)$ for which $1 \leq x_1 \leq 1 + t$ and $1 \leq x_i \leq (1 + t)x_{i-1}$ for $i = 2,\ldots,n$. The main idea of the proof of Theorems~\ref{thm6} and~\ref{thm7} is to divide these inequalities into ``narrower'' inequalities.  We state this precisely in the following example, and then in full generality.

\begin{exa}{\rm \small
Since $1 \leq x_2 \leq (1 + t)x_1$ and $(1 + t)x_1 \geq 1 + t$, we have either $1 \leq x_2 \leq 1 + t$ or $1 + t \leq x_2 \leq (1 + t) x_1$. If $1 \leq x_2 \leq 1 + t$, then either $1 \leq x_3 \leq 1 + t$ or $1 + t \leq x_3 \leq (1 + t) x_2$. On the other hand, if $1 + t \leq x_2 \leq (1 + t) x_1$, then $(1 + t)x_2 \geq (1 + t)^2$, so we have $1 \leq x_3 \leq 1 + t$, $1 + t \leq x_3 \leq (1 + t)^2$ or $(1 + t)^2 \leq x_3 \leq (1 + t)x_2$. The following table presents all such choices for $n = 4$.

\renewcommand{\arraystretch}{1.2}
$$
\begin{array}{|c|c|c|c|}
 \hline
 1 \leq x_1 \leq 1 \hspace{-2pt} + \hspace{-2pt} t & 1 \leq x_2 \leq 1 \hspace{-2pt} + \hspace{-2pt} t & 1 \leq x_3 \leq 1 \hspace{-2pt} + \hspace{-2pt} t & 1 \leq x_4 \leq 1 \hspace{-2pt} + \hspace{-2pt} t \\
 \cline{4-4}
 & & & 1 \hspace{-2pt} + \hspace{-2pt} t \leq x_4 \leq (1 \hspace{-2pt} + \hspace{-2pt} t)x_3 \\
\cline{3-4}
 & & 1 \hspace{-2pt} + \hspace{-2pt} t \leq x_3 \leq (1 \hspace{-2pt} + \hspace{-2pt} t)x_2 & 1 \leq x_4 \leq 1 \hspace{-2pt} + \hspace{-2pt} t \\
\cline{4-4}
 & &  & 1 \hspace{-2pt} + \hspace{-2pt} t \leq x_4 \leq (1 \hspace{-2pt} + \hspace{-2pt} t)^2 \\
\cline{4-4}
& &  & (1 \hspace{-2pt} + \hspace{-2pt} t)^2 \leq x_4 \leq (1 \hspace{-2pt} + \hspace{-2pt} t)x_3 \\
\cline{2-4}
& 1 \hspace{-2pt} + \hspace{-2pt} t \leq x_2 \leq (1 \hspace{-2pt} + \hspace{-2pt} t)x_1 & 1 \leq x_3 \leq 1 \hspace{-2pt} + \hspace{-2pt} t & 1 \leq x_4 \leq 1 \hspace{-2pt} + \hspace{-2pt} t \\
\cline{4-4}
& & & 1 \hspace{-2pt} + \hspace{-2pt} t \leq x_4 \leq (1 \hspace{-2pt} + \hspace{-2pt} t)x_3 \\
\cline{3-4}
& & 1 \hspace{-2pt} + \hspace{-2pt} t \leq x_3 \leq (1 \hspace{-2pt} + \hspace{-2pt} t)^2 & 1 \leq x_4 \leq 1 \hspace{-2pt} + \hspace{-2pt} t \\
\cline{4-4}
& & & 1 \hspace{-2pt} + \hspace{-2pt} t \leq x_4 \leq (1 \hspace{-2pt} + \hspace{-2pt} t)^2 \\
\cline{4-4}
& & & (1 \hspace{-2pt} + \hspace{-2pt} t)^2 \leq x_4 \leq (1 \hspace{-2pt} + \hspace{-2pt} t)x_3 \\
\cline{3-4}
& & (1 \hspace{-2pt} + \hspace{-2pt} t)^2 \leq x_3 \leq (1 \hspace{-2pt} + \hspace{-2pt} t)x_2 & 1 \leq x_4 \leq 1 \hspace{-2pt} + \hspace{-2pt} t \\
\cline{4-4}
& & &  1 \hspace{-2pt} + \hspace{-2pt} t \leq x_4 \leq (1 \hspace{-2pt} + \hspace{-2pt} t)^2 \\
\cline{4-4}
& & & (1 \hspace{-2pt} + \hspace{-2pt} t)^2 \leq x_4 \leq (1 \hspace{-2pt} + \hspace{-2pt} t)^3 \\
\cline{4-4}
& & & (1 \hspace{-2pt} + \hspace{-2pt} t)^3 \leq x_4 \leq (1 \hspace{-2pt} + \hspace{-2pt} t)x_3 \\
\hline

\end{array}$$
}
\end{exa}

\begin{lemma}
 The $t$-Cayley polytope $\pC_n(t)$ can be subdivided into polytopes defined by inequalities for variables $x_1,\ldots,x_n$ so that:
 \begin{itemize}
  \item[I1] the inequalities for $x_1$ are $1 \leq x_1 \leq 1+t$;
  \item[I2] the inequalities for each $x_i$ are either $(1+t)^{j_i} \leq x_i \leq (1+t)^{j_i+1}$ or $(1+t)^{j_i} \leq x_i \leq (1+t)x_{i-1}$ (only if $i \geq 2$) for some $j_i \geq 0$;
  \item[I3] for $i \geq 2$, we have  $j_i \leq j_{i - 1} + 1$, and $j_i = j_{i-1} + 1$ if and only if the inequalities for $x_i$ are $(1+t)^{j_i} \leq x_i \leq (1+t)x_{i-1}$.
  \end{itemize}
\end{lemma}
\begin{proof}
 It is clear that the polytopes defined by inequalities I1--I3 have volume $0$ intersections. Let us prove by induction that each point of $\pC_n(t)$ lies in one of the polytopes.
 For $n = 1$, this is clear, assume that the statement holds for $n-1$. For a point $(x_1,\ldots,x_n) \in \pC_n(t)$, we have $(x_1,\ldots,x_{n-1}) \in \pC_{n-1}(t)$, and $1 \leq x_n \leq (1+t)x_{n-1}$. By induction, $(1+t)^{j_{n-1}} \leq x_{n-1} \leq (1+t)^{j_{n-1}+1}$ or $(1+t)^{j_{n-1}} \leq x_{n-1} \leq (1+t)x_{n-2}$. Note that $(1+t)x_{n-1} \geq (1+t)^{j_{n-1}+1}$. Thus at least one (and at most two) of the statements $1 \leq x_n \leq 1+t$, $1+t \leq x_n \leq (1+t)^2,\ldots,(1+t)^{j_{n-1}} \leq x_n \leq (1+t)^{j_{n-1}+1}$, $(1+t)^{j_{n-1}+1} \leq x_n \leq (1+t)x_{n-1}$ is true. in other words, we can either choose $0 \leq j_n \leq j_{n-1}$ so that $(1+t)^{j_n} \leq x_n \leq (1+t)^{j_n+1}$, or we have $(1+t)^{j_n} \leq x_n \leq (1+t)x_n$ for $j_n = j_{n-1}+1$. This finishes the inductive step.
\end{proof}

We claim the the polytopes constructed in the lemma are precisely the polytopes $\pD_T(t)$ from Section~\ref{t}. So say that we have inequalities satisfying the conditions I1-I3 and defining a polytope $\s P$. Our goal is to construct the unique plane tree $T$ satisfying $\pD_T(t) = \s P$.

\smallskip

Assume that $k$, $1 \leq k \leq n$, is the largest integer so that the inequalities for $x_i$, $i=2,\ldots,k$, are of the form $(1+t)^{i-1} \leq x_i \leq (1+t)x_{i-1}$. In particular, the inequalities for $x_{k+1}$ are \emph{not} of the form $(1+t)^k \leq x_{k+1} \leq (1+t)x_k$, but instead $(1+t)^{j_{k+1}} \leq x_{k+1} \leq (1+t)^{j_{k+1}+1}$ for some $j_{k+1}$, $0 \leq j_{k+1} \leq k-1$.

\smallskip

There exist unique integers $a_1,\ldots,a_k \geq 1$, $a_1+\ldots+a_k = n-k$, satisfying the following properties:
\begin{itemize}
 \item $j_{k+1},\ldots,j_{k + a_1} \geq k-1$, $j_{k+a_1+1} < k-1$;
 \item $j_{k+a_1+1},\ldots,j_{k+a_1+a_2} \geq k-2$, $j_{k+a_1+a_2+1} < k-2$;
 \item $j_{k+a_1+a_2+1},\ldots,j_{k+a_1+a_2+a_3} \geq k-3$, $j_{a_1+a_2+a_3+1} < k-3$;
 \item etc.
\end{itemize}
Note that if $a_1 \geq 1$, then $j_{k+1} = k-1$, if $a_2 \geq 1$, then $j_{k+a_1+1} = k-2$, etc.

\smallskip

In other words, among the inequalities for $x_{k+1},\ldots,x_n$, the first $a_1$ inequalities have at least $(1+t)^{k-1}$ on the left, the next $a_2$ inequalities have at least $(1+t)^{k-2}$ on the left, etc. Say that among the inequalities for $x_{k+1},\ldots,x_n$, the first $a_1$ inequalities define the polytope $(1+t)^{k-1}\s P_1$, the next $a_2$ inequalities define the polytope $(1+t)^{k-2}\s P_2$, etc. By induction, the polytopes $\s P_1,\ldots,\s P_k$ are of the form $\pD_{T_1},\ldots,\pD_{T_k}$ for some plane trees $T_1,\ldots,T_k$ on $a_1+1,a_2+1,\ldots,a_k+1$ nodes. Define the tree $T$ by taking a root with $k$ successors and subtrees $T_1,\ldots,T_k$.

\begin{exa} {\rm \small
 Say that
 $$\s P = \left\{\begin{array}{lll}
             (x_1,\ldots,x_{11}) \colon & 1 \leq x_ 1 \leq 1+t, & 1+t \leq x_ 2 \leq (1+t)\ts x_ 1, \\
	  (1+t)^2 \leq x_3 \leq (1+t)\ts x_ 2, & (1+t)^2 \leq x_4 \leq (1+t)^3, & (1+t)^3 \leq x_ 5 \leq (1+t)\ts x_4, \\
           (1+t)^3 \leq x_ 6 \leq (1+t)^4, & (1+t)^3 \leq x_ 7 \leq (1+t)^4, & (1+t)^4 \leq x_8 \leq (1+t)\ts x_ 7, \\
            1+t \leq x_ 9 \leq (1+t)^2, &(1+t)^2 \leq x_{10} \leq (1+t)\ts x_ 9,&1 \leq x_{11} \leq 1+t
            \end{array}\right\}.$$
 We have $k = 3$ and $a_1 = 5$, $a_2 = 2$, $a_3 = 1$. Furthermore,
$$\begin{array}{ll}
\s P_1 & =\left\{\begin{array}{lll}
		    (x_1,\ldots,x_5) \colon & 1 \leq x_1 \leq 1+t, & 1+t \leq x_2 \leq (1+t)\ts x_1, \\
                    1+t \leq x_3 \leq (1+t)^2, & 1+t \leq x_4 \leq (1+t)^2, & (1+t)^2 \leq x_5 \leq (1+t)\ts x_4
                   \end{array} \right\},\\

\s P_2 & = \set{(x_1,x_2) \colon \quad 1 \leq x_1 \leq 1+t, \quad  1+t \leq x_2 \leq (1+t)x_1},\\
\s P_3 & = \set{x_1 \colon  \quad 1 \leq x_1 \leq 1+t}.
\end{array}$$

\noindent
The corresponding subtrees $T_1,T_2,T_3$ of  the tree $T$ is shown with full lines in Figure~\ref{fig13}.
\begin{figure}[ht!]
 \begin{center}
   \includegraphics[height=2.5cm]{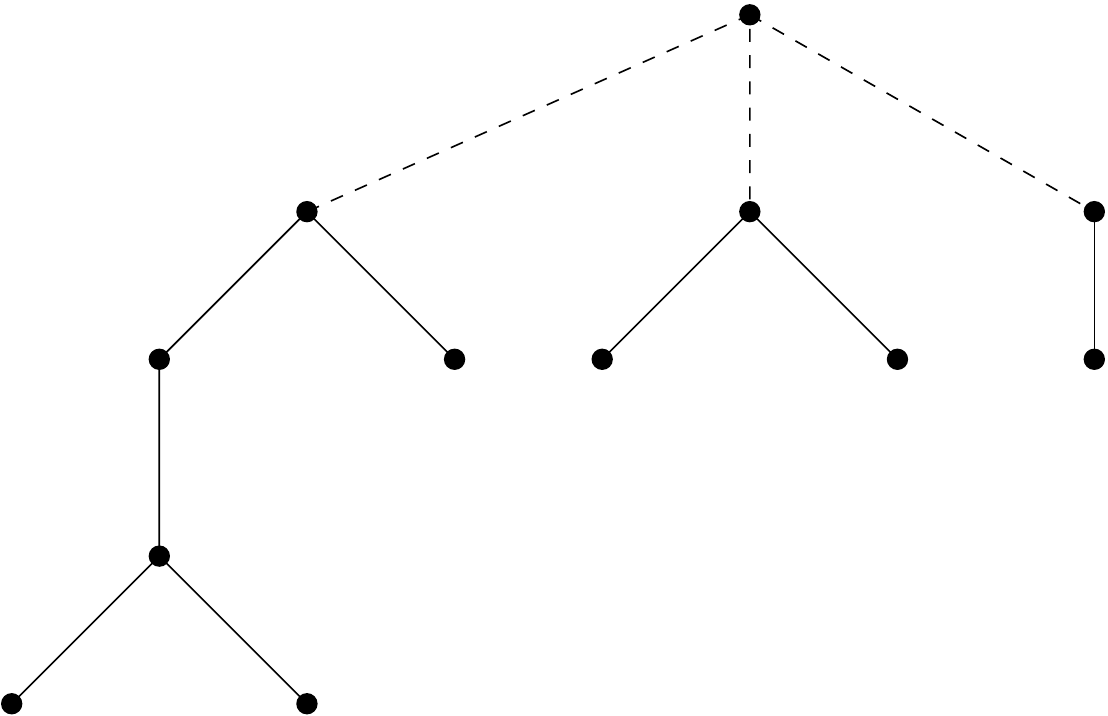}
   \caption{The plane tree corresponding to a subpolytope.}
   \label{fig13}
 \end{center}
\end{figure}
}
\end{exa}

\begin{lemma}
 We have $\pD_T(t) = \s P$, and $T$ is the only tree with this property.
\end{lemma}
\begin{proof}
 The root has $k$ successors, hence the inequalities for $x_1,\ldots,x_k$ determined by $T$ are $1 \leq x_k/(1+t)^{k-1} \leq x_{k-1}/(1+t)^{k-2} \leq \ldots \leq x_2/(1+t) \leq x_1 \leq 1+t$, which is equivalent to $1 \leq x_1 \leq 1+t$, $(1+t)^{i-1} \leq x_i \leq (1+t)x_{i-1}$ for $i = 2,\ldots,k$. By induction, $\pD_{T_i}(t) = P_i$, and so $\pD_T(t) = \s P$. For the second part,
 if $\pD_T(t) = \pD_{T'}(t)$ for plane trees $T$ and $T'$, then the degree sequences of $T$ and $T'$ have to be the same, and so $T = T'$.
\end{proof}

The lemma shows that $\set{\pD_T(t) \colon T \mbox{ a plane tree on }n+1 \mbox{ nodes}}$ is a subdivision of the polytope $\pC_n(t)$.  This implies the second part of Theorem~\ref{thm7}. \hfill $\sq$

\smallskip


Recall that each $\pD_T(t)$ is subdivided into $\binom{n}{d_1,d_2,\ldots}$ simplices $\pS_{T'}$ enumerated by labeled trees $T'$ which become $T$ if we erase the labels. This completes the proof of the second part of Theorem~\ref{thm6}.   \hfill  $\sq$

\subsection{Subdivisions of the $t$-Gayley polytope}

\begin{proof}[Proof of the first part of Theorem~\ref{thm8}]
Consider a transformation
$$
A \colon x_{i_k} \mapsto t(x_{i_k}(1+t)^{j_k}+x_{l_k})
$$
for $k$ corresponding to nodes that do not have the maximal label in their connected component.
Then $\pS_F(t)$ is mapped into the standard simplex with volume $1/n!$, and $A$ is the
composition of a linear transformation with an upper triangular matrix in the standard basis
and a translation. The determinant of the linear transformation is
$$
t^{n+1-k(F)} (1+t)^{\alpha(F)} = t^{|E(F)|} (1+t)^{\alpha(F)}.
$$
On the other hand, if the components of $F$ are trees $T_1,\ldots,T_m$, then
$$\sum_{\Phi(G) = F} t^{|E(G)|} = \prod_{j=1}^m \sum_{\Phi(G_j) = T_j} t^{|E(G_j)|}
= \prod_{j=1}^m t^{|E(T_i)|} (1+t)^{\alpha(T_j)} = t^{|E(F)|} (1+t)^{\alpha(F)},
$$
as desired.
\end{proof}

To prove the first part of Theorem~\ref{thm9}, we need the following elementary lemma.

\begin{lemma} \label{cone}
 If $\s P \subset \R^n$ has volume, then $\cone(\s P)$ has volume, and
 $$\vol(\cone(\s P)) = \frac 1{n+1} \vol(\s P).$$
\end{lemma}

\begin{proof}[Proof of the first part of Theorem~\ref{thm9}]
By the definition of $\pD_F(t)$, induction on the number of components of a plane forest $F$, and Lemma~\ref{cone}, we have
$$n!\ts\vol(\pD_F(t)) = \frac{n!}{\prod_{j=2}^m(a_j+\ldots+a_m)} \prod_{j=1}^m \vol(\pD_{T_j}(t)),$$
where the components of $F$ are $T_1,\ldots,T_m$ and $T_j$ has $a_j$ nodes. If $m = 2$, the reduced degree sequence of $T_1$ is $(d_1,\ldots,d_{a_1-1})$ and the reduced degree sequence of $T_2$ is $(d_{a_1},\ldots,d_{a_1+a_2-2})$, then
$$\alpha(T_1) + \alpha(T_2) = \binom{a_1}2 - \sum_{i=1}^{a_1-1} i d_i + \binom{a_2}2 - \sum_{i=1}^{a_2-1}i d_{i+a_1-1} =$$
$$=\binom{a_1}2 - \sum_{i=1}^{a_1-1} i d_i + \binom{a_2}2 - \sum_{i=a_1}^{a_1+a_2-2}(i-a_1+1) d_{i} = $$
$$=\binom{a_1}2 + \binom{a_2}2  + (a_1-1)(a_2-1) - \sum_{i=1}^{a_1+a_2-2} i d_i = \binom{a_1+a_2-1}2 - \sum_{i=1}^{a_1+a_2-2} i d_i.$$
The proof of
$$\alpha(T_1) + \ldots + \alpha(T_m) = \binom{n+2-m}2 - \sum_{i=1}^{n+1-m} i d_i$$
for $m \geq 3$ is a simple extension of this argument. Since $\sum_i d_i = |E(F)|$, this finishes the proof.
\end{proof}

\begin{lemma}
 For a labeled $($respectively, plane$)$ forest $F$ on $n+1$ vertices, $\s S_F(t) \subseteq \s G_n(t)$ $($respectively, $\s D_F(t) \subseteq \s G_n(t) \ts)$. \qed
\end{lemma}

The proof is very similar to the proof of Lemma \ref{lemma1} and we omit it.

\begin{proof}[Proof of Proposition~\ref{prop2}]
 If $v$ is a node in the left-most tree with successors $v_1,\ldots,v_k$, then the inequalities for the corresponding $x_i,\ldots,x_{i+k-1}$ come from $\pD_{T_1}(t)$ and are
 $$1 \leq x_{i+k-1}/(1+t)^{j+k-1} \leq \ldots \leq x_i/(1+t)^j \leq 1+t.$$
 If we subtract $x_0 = 1$, we get precisely $0 \leq c(v_1,F;t) \leq \ldots \leq c(v_k,F;t) \leq t$. If $v$ is a node in a different tree and its successors are $v_1,\ldots,v_k$, then we take inequalities
 $$1 \leq x_{i+k-1}/(1+t)^{j+k-1} \leq \ldots \leq x_i/(1+t)^j \leq 1+t$$
 and ``cone'' them, i.e.\ replace $x_i,\ldots,x_{i+k-1}$ by $x_i/x_l,\ldots,x_{i+k-1}/x_l$. Multiplying by $x_l$ and subtracting $x_l$ yields
 $$0 \leq x_{i+k-1}/(1+t)^{j+k-1} - x_l \leq \ldots \leq x_i/(1+t)^j \leq tx_l,$$
 which is the same as
 $$0 \leq c(v_1,F;t) \leq \ldots \leq c(v_k,F;t) \leq c(w,F;t).$$
 If we ``cone'' the inequalities again, they do not change.  If $w$ is the root of a tree that is not the left-most component, the inequality for the corresponding component $x_l$ is $0 \leq x_l \leq x_{l'}$, where $l$ (respectively $l'$) is the position in NFS of $w$ (respectively, of the root of the tree to the left). Multiplying by $t$ gives $0 \leq c(w,F;t) \leq c(w',F;t)$.
\end{proof}

\begin{proof}[Proof of the second part of Theorem~\ref{thm9}]
For the second part, take $(x_1,\ldots,x_n) \in \pG_n(t)$, and let $k$, $1 \leq k \leq n$, be the largest integer for which $x_k \geq 1$. In particular, $(x_1,\ldots,x_k) \in \pC_k(t)$, and if $k < n$, $0 \leq x_{k+1} \leq 1$ and $(x_{k+2},\ldots,x_n) \in x_{k+1} \s G_{n-k-1}(t)$. By induction, Theorem~\ref{thm7} and the definition of the cone, this means that
$$ (x_1,\ldots,x_n) \in \pD_{T_1}(t) \times \cone(\pD_{T_2}(t) \times \cone(\pD_{T_3}(t) \times \cdots)),$$
i.e.\ $\set{\pD_F(t) \colon F \mbox{ plane forest on } n+1 \mbox{ nodes}}$ is a subdivision of the $t$-Gayley polytope.
\end{proof}

\begin{proof}[Proof of the second part of Theorem~\ref{thm8}]
%

The polytope $\pD_F(t)$ is determined by ordering the coordinates of the nodes with the same parent. Choosing an ordering of all nodes (i.e.\ changing the plane forest $F$ into a labeled forest $F'$) produces a simplex $\pS_{F'}(t)$. This finishes the proof of Theorem~\ref{thm8}.
\end{proof}

\medskip

\subsection{Subdivisions of the Tutte polytope} \label{proofs:tutte}

As mentioned in the introduction to this section, some of the proofs for the results for the Tutte polytope are only sketches. We essentially follow the proof of the results for the $t$-Gayley polytope.

\smallskip

\begin{proof}[Proof of the first part of Theorem \ref{thm10}] We are given a labeled forest $F$. Consider the transformation
$$
B_F \colon x_{i_k} \mapsto x_{i_k} + (1-q)(1-x_{l_k})\ts,
$$
with the inverse
$$
B_F^{-1} \colon x_{i_k} \mapsto x_{i_k} - \frac{(1-q)(1-x_{l_k})}{q}\..
$$
Here $i_k$ is the position in NFS of the node with label $k$, and $l_k$ is the position of the node with the maximal label in the same component. If $k$ is the label of a root, then $l_k = i_k$ and hence
$$B_F \colon x_{i_k} \mapsto qx_{i_k} + 1 - q\ts,$$
$$B_F^{-1} \colon x_{i_k} \mapsto (x_{i_k} - 1 + q)/q \ts,$$
This means that $B_F$ is the composition of a translation and a linear transformation with an upper triangular matrix in the standard basis. Moreover, the diagonal elements are $q$ (for coordinates corresponding to roots, so there are $k(F)-1$ of them) and $1$ (for other nodes). In other words, the simplex $\s S_F(t)$ with volume $t^{|E(F)|}(1+t)^{\alpha(F)}/n!$ is mapped into a simplex with volume $q^{k(F)-1}t^{|E(F)|}(1+t)^{\alpha(F)}/n!$. The coordinate $c(i,F;t) = tx_l$ is mapped into $t(x_l-1+q)/q$, and the coordinate $c(i,F;t) = x_i/(1+t)^j - x_l$ is mapped into
$$\frac{x_i - \frac{(1-q)(1-x_{l_k})}{q} }{(1+t)^j} - \frac{x_l-1+q}q\..
$$
Multiplying this by $q$, we obtain  $c(i,F;q,t)$. Therefore, $B_F(\s S_F(t)) = \s S_F(q,t)$, as desired.
\end{proof}

\begin{lemma}
  For a labeled $($respectively, plane$)$ forest $F$ on $n+1$ vertices, $\s S_F(q,t) \subseteq \s T_n(q,t)$ $($respectively, $\s D_F(q,t) \subseteq \s T_n(q,t) \ts)$. 
\end{lemma}
\begin{proof}
 Since
 $$\frac{qx_i-(1-q)(1-x_l)}{(1+t)^j} - (x_l-1+q) \geq 0,$$
 for some $j \geq 0$ and $l$, we have $x_i \geq x_l$, where $t(x_l-1+q)$ is the coordinate of the root of the same component. But $t(x_l-1+q) \geq 0$, so $x_i \geq x_l \geq 1-q$.
 Assume that $v$ is the $(i-1)$-st visited node and $v'$ the $i$-th visited node, where $i \geq 2$. We have the following cases:
 \begin{itemize}
  \item $v'$ and $v$ have the same parent and $v'$ has a smaller label;
  \item $v'$ is the child of $v$ with the largest label;
  \item the unique path from $v$ to $v'$ goes up at least once, then down right, and then down to the child with the largest label;
  \item $v$ is the last node visited in NFS in its component, and $v'$ has the largest label among the remaining nodes.
 \end{itemize}
 Assume that $v'$ and $v$ have the same parent and $v'$ has a smaller label; in this case we have
 $$0 \leq \frac{qx_i-(1-q)(1-x_l)}{(1+t)^{j+1}} - (x_l-1+q) \leq \frac{qx_{i-1}-(1-q)(1-x_l)}{(1+t)^j} - (x_l-1+q) \leq t(x_l-1+q)$$
 for $l$ corresponding to the node with the maximal label in the same component as $v$ and $v'$. The middle inequality gives
 $$qx_i \leq q(1+t)x_{i-1} - t(1-q)(1-x_l).$$
 We already know that if $x_{j-1}$ belongs to the same component, then $x_{j-1} \geq x_l$, so
 $$qx_i \leq q(1+t)x_{i-1} - t(1-q)(1-x_{j-1}).$$
 If $x_{j-1}$ belongs to a component to the left, then $x_{j-1} \geq x_{l'} \geq x_l$, so we have $qx_i \leq q(1+t)x_{i-1} - t(1-q)(1-x_{j-1})$ in this case as well.  We omit the rest of the proof.
\end{proof}

\begin{proof}[Proof of Proposition \ref{prop3}] Note that $B_F$ is the same for all labeled forests with the same underlying plane forest. In light of the above proof, it is enough to prove that $B(\s D_F(t)) = \s D_F(q,t)$, i.e.\ that
$$B_F\left( \pD_{T_1}(t) \times \cone(\pD_{T_2}(t) \times \cone(\pD_{T_3}(t) \times \cdots))\right) = $$
$$ = \pD_{T_1}(t) \times \cone_q(\pD_{T_2}(t) \times \cone(\pD_{T_3}(t) \times \cdots)).$$
Suppose that $T_1$ has $k$ nodes. Then $B_F \colon x_i \mapsto x_i + (1-q)(1-x_0) = x_i$ for $1 \leq i \leq k-1$. The inequality $0 \leq x_l \leq x_{l'}$ is transformed to $1-q \leq x_l \leq x_{l'}$ via $B_F$, which is also the inequality we get when using $\cone_q$. The inequalities for $\s D_{T_p}(t)$, $p \geq 2$, have terms of the form $x_i/(1+t)^j$. If we ``cone'' these inequalities, we get terms of the form $x_i/((1+t)^jx_l$, and ``coning'' again does not change them. Applying $B_F$ gives
$$\frac{x_i - \frac{(1-q)(1-x_l)}q}{(1+t)^j \frac{x_l-1+q}q} = \frac{qx_i - (1-q)(1-x_l)}{(1+t)^j(x_l-1+q)},$$
which is also the effect of using $\cone_q$.
\end{proof}

\begin{proof}[Proof of Theorem \ref{thm12}] The first part follows from the previous two proofs. Now take $(x_1,\ldots,x_n) \in \s T_n(q,t)$, and let $k$, $1 \leq k \leq n$, be the largest integer for which $x_k \geq 1$. In particular, $(x_1,\ldots,x_k) \in \pC_k(t)$, and if $k < n$, $1-q \leq x_{k+1} \leq 1$, and we claim that and $q(x_{k+2},\ldots,x_n) - (1-q)(1-x_{k+1}) \in (x_{k+1} - 1 + q ) \s T_{n-k-1}(q,t)$. By induction, Theorem~\ref{thm7} and the definition of $\cone_q$ this will imply that
$$ (x_1,\ldots,x_n) \in \pD_{T_1}(t) \times \cone_q(\pD_{T_2}(t) \times \cone_q(\pD_{T_3}(t) \times \cdots)),$$
i.e.\ $\set{\pD_F(q,t) \colon F \mbox{ plane forest on } n+1 \mbox{ nodes}}$ is a subdivision of the $t$-Gayley polytope.
The inequality
$$\frac{qx_n - (1-q)(1-x_{k+1})}{x_{k+1}-1+q} \geq 1-q$$
is equivalent to
$$x_n \geq 1-q$$
for $q > 0$; the inequality
$$q \frac{qx_i - (1-q)(1-x_{k+1})}{x_{k+1}-1+q} \leq $$
$$\leq q(1+t) \frac{qx_{i-1} - (1-q)(1-x_{k+1})}{x_{k+1}-1+q} - t(1-q)\left( 1 - \frac{qx_{j-1} - (1-q)(1-x_{k+1})}{x_{k+1}-1+q}\right)$$
for $j > k + 2$ is equivalent to
$$qx_i \leq q(1+t)x_{i-1} - t(1-q)(1-x_{j-1}) + t(1-q) x_{k+1}$$
and follows from $qx_i \leq q(1+t)x_{i-1} - t(1-q)(1-x_{j-1})$; and the inequality
$$q \frac{qx_i - (1-q)(1-x_{k+1})}{x_{k+1}-1+q} \leq $$
$$\leq q(1+t) \frac{qx_{i-1} - (1-q)(1-x_{k+1})}{x_{k+1}-1+q} - t(1-q)\left( 1 - \frac{qx_{j-1} - (1-q)(1-x_{k+1})}{x_{k+1}-1+q}\right)$$
for $j = k + 2$ is equivalent to
$$q x_i \leq (1+t)q x_{i-1} - t(1-q)(1-x_{k+1}),$$
which is given. This completes the proof.
\end{proof}

\begin{proof}[Proof of the second part of Theorem \ref{thm10}]
 This follows in the same way as the second part of Theorem \ref{thm8} followed from the second part of Theorem \ref{thm9}.
\end{proof}

\medskip

\subsection{Vertices}
Here we prove the results from Section~\ref{vertices}.

\begin{proof}[Proof of Proposition~\ref{vertsx}]
Pick $t > 0$ and a labeled forest $F$ on $n+1$ nodes. Then $v_p(F;t)$, the $p$-th vertex of the simplex $\pS_F(t)$, $1 \leq p \leq n+1$, is the (unique) solution of the system of equations
$$
c(1,F;t) = \ldots = c(p-1,F;t) = 0, \quad c(p,F;t) = \ldots = c(n,F;t) = t\ts.
$$
It is easy to check that the solution agrees with the statement of the proposition.
\end{proof}

\begin{proof}[Proof of Proposition~\ref{vertcayley}]
For $(x_1,\ldots,x_n)$ to be a vertex of $\pC_n(t)$, one of the inequalities $1 \leq x_i$ and $x_i \leq 2x_{i-1}$ (where $x_0=1$) must be an equality for every $i$. That means that we have $x_1 \in \set{1,1+t}$, $x_i \in \set{1,(1+t)x_{i-1}}$ for $i = 2,\ldots,n$. This completes the proof. \end{proof}

\begin{proof}[Proof of Proposition~\ref{vertqsx}]
Recall the construction of $B_F$ from Subsection \ref{proofs:tutte}. For $p \leq r$, we have $x_l + (1-q)(1-x_l) = 1$, and for $p > r$, we have $x_l + (1-q)(1-x_l) = 1-q$. Similarly, $x_i + (1-q)(1-x_l) = x_i$ for $p \leq r$, and $x_i + (1-q)(1-x_l) = 1-q$ for $p < r$. This proves the proposition. \end{proof}

\begin{proof}[Proof of Theorem~\ref{verttutte}] \. For $q = 0$, this is just Proposition \ref{vertcayley}. Assume $q > 0$. By Proposition~\ref{vertqsx}, the Tutte polytope $\s T_n(q,t)$ is the convex hull of certain points $v = (x_1,\ldots,x_n)$ which have the following properties:
\begin{itemize}
 \item for every $i \geq 1$, $x_i$ is either $(1+t)^j$ for $j \geq 0$ or $1-q$;
 \item for every $i \geq 1$, if $x_{i-1} = (1+t)^j$ and $x_i = (1+t)^{j'}$, then $j' \leq j + 1$ (in particular, $x_1$ is either $1+t$, $1$ or $1-q$);
 \item if $x_i = 1-q$, then $x_{i+1} = \ldots = x_n = 1-q$.
\end{itemize}
We want to see that every such vertex is in the convex hull of $V_n(q,t)$. Suppose that $x_1,\ldots,x_k \neq 1-q$ and $x_{k+1} = \ldots = x_n = 1-q$. Then $(x_1,\ldots,x_k) \in \s C_k(t)$, and therefore it is a convex combination of points in $V_k(t)$. Therefore $(x_1,\ldots,x_n)$ is in the convex hull of the set $V'_n(q,t)$ that we get if we replace \emph{some} (i.e.\ not necessarily all) of the trailing $1$'s of the points in $V_n(t)$ by $1-q$. Take a point $(x_1,\ldots,x_n)$ that has $x_k = 1, x_{k+1} = \ldots = x_n = 1-q$. Then it is on the line between $(x_1,\ldots,x_{k-1},(1+t)x_{k-1},1-q,\ldots,1-q)$ and $(x_1,\ldots,x_{k-1},1-q,1-q,\ldots,1-q)$. This implies that $(x_1,\ldots,x_n)$ is in the convex hull of $V_n(q,t)$. 

\smallskip

It remains to prove that no point in $V_n(q,t)$ can be expressed as a convex combination of the others. For $S \subseteq \set{1,\ldots,n}$, define $x^S$ to be the element of $V_n(q,t)$ that satisfies $x_i = (1+t)x_{i-1} \iff i \in S$. For example, for $n = 4$, $x^\emp = (1-q,1-q,1-q,1-q)$, $x^{\set{1,3}} = (1+t,1,1+t,1-q)$ and $x^{\set{2,3,4}} = (1,1+t,(1+t)^2,(1+t)^3)$.  We need the following lemma.

\begin{lemma} \label{lemma2}
 If $S \neq T$, there exists a defining inequality $H(x) \leq 0$ of $\s T_n(q,t)$ so that $H(x^S) = 0$ and $H(x^T) < 0$.
\end{lemma}

If $x^S = \sum_{R \neq S} \alpha_R x^R$ for $\alpha_R \geq 0$, $\sum_{R \neq S} \alpha_R = 1$, take $T$ with $\alpha_T > 0$ and $H$ from the lemma. Then $0 = H(x^S) = \sum_{R \neq S} \alpha_R H(x^R) \leq \alpha_T H(x^T) < 0$. The contradiction proves that the vertices of $\s T_n(q,t)$ are exactly the points in $V_n(q,t)$, and completes the proof of the theorem. \end{proof}

\begin{proof}[Proof of Lemma~\ref{lemma2}] Assume first that $S \not\subseteq T$. For $i \in S \setminus T$ and $j = 1$ we have
$$qx^S_i = q (1+t) x^S_{i-1} = q(1+t)x^S_{i-1} - t(1-q)(1-x^S_{j-1})$$
and
$$qx^T_i < q (1+t) x^T_{i-1} = q(1+t)x^T_{i-1} - t(1-q)(1-x^T_{j-1}).$$
We can now assume that $S \subset T$. Let $j = \min(T \setminus S) + 1$. If there is $i \geq j$ in $S$ (and also $i \in T$), then $x^S_{j-1} = 1$ and $x^T_{j-1} > 1$. Therefore
$$qx^S_i = q (1+t) x^S_{i-1} = q(1+t) x^S_{i-1} - t(1-q)(1-x_{j-1})$$
and
$$qx^T_i = q (1+t) x^T_{i-1} < q(1+t) x^T_{i-1} - t(1-q)(1-x_{j-1}).$$
Otherwise, $\max S < \max T$. If $\max T = n$, then $x^S_n = 1-q$ and $x^T_n \geq 1 > 1-q$. If $\max T \leq n-1$, take $i = j = \max T + 1 \leq n$. Then $x^S_i = x^S_{i-1} = x^S_{j-1} = 1-q$ and
$$qx^S_i = q(1-q) = q(1+t)x^S_{i-1} - t(1-q)(1-x^S_{j-1}),$$
while $x^T_i = 1-q$, $x^T_{i-1} = x^T_{j-1} \geq 1 > 1-q$ and
$$qx^T_i = q(1-q) < q(1+t)x^T_{i-1} - t(1-q)(1-x^T_{j-1}) = (q+t) x^T_{i-1} - t(1-q).$$
This finishes the proof of the lemma.
\end{proof}

\medskip

\section{Final remarks and open problems}\label{s:fin}

\subsection{} \label{s:fin-1} \.
By now, there are quite a few papers on ``combinatorial volumes'',
i.e.~expressing combinatorial sequences as volumes of certain
polytopes.  These include \emph{Euler numbers} as volumes of
hypersimplices~\cite{S1} (see also~\cite{ABD,ERS,LP,Pos}),
\emph{Catalan numbers}~\cite{GGP},
\emph{Cayley numbers} as volumes of permutohedra (see~\cite{Pak,Ziegler}),
the \emph{number of linear extensions of posets}~\cite{S2}, etc.

Let us mention a mysterious connection of our results to those in~\cite{SP},
where the number of (generalized) \emph{parking functions} appears as the volume
of a certain polytope, which is also combinatorially equivalent to an $n$-cube.
The authors observe that in a certain special case, their polytopes have 
(scaled) volume the \emph{inversion polynomial}~$\IP_n(t)$, compared to
$t^n \IP_n(1+t)$ for the $t$-Cayley polytopes.  The connection between
these two families of polytopes is yet to be understood, and the
authors intend to pursue this in the future.

In this connection, it is worth noting that Theorem~\ref{t:main} and
our triangulation construction
seem to be fundamentally about labeled trees rather than parking functions,
since the full Tutte polynomial $\rT_{K_n}(q,t)$ seems to have no
known combinatorial interpretation in the context of parking
functions (cf.~\cite{Stanley,Hag}).
Curiously, the specialization $\rT_G(1,t)$ has a natural combinatorial
interpretation for $G$-parking functions for general graphs~\cite{CL}.


\subsection{} \label{s:fin-2} \. It is worth noting that all
simplices in the triangulation of the Cayley polytopes are
Schl\"{a}fli orthoschemes, which play an important role in
combinatorial geometry.  For example, in McMullen's
\emph{polytope algebra} (which formalizes properties of
\emph{scissor congruence}),
orthoschemes form a linear basis~\cite{McM} (see also~\cite{Dup,Pak}).
Moreover, Hadwiger's conjecture states that \emph{every} convex
polytope in $\rr^d$ can be triangulated into a finite number of
orthoschemes~\cite{Had} (see also~\cite{BKKS}).

Let us emphasize here that not all simplices of triangulations
constructed in Sections~\ref{gayley},~\ref{t} and~\ref{s:tutte} are
orthoschemes.  Let us also mention that triangulations of
polytopes $\pD_T$ and $\pD_F$ given by $\pS_T$ and $\pS_F$ are
the usual \emph{staircase triangulations} of the products of
simplices (see e.g.~\cite[$\S 6.2$]{DRS}).

\subsection{} \label{s:fin-3} \.
In a follow-up note~\cite{KP}, we prove Cayley's theorem (Theorem~\ref{t:cayley})
by an explicit volume-preserving map, mapping integer points in~$\pC_n$ into
a simplex corresponding to integer partitions as in Theorem~\ref{t:cayley}, a
rare result similar in spirit to~\cite{PV}.  As an application of
our Theorem~\ref{t:pol}, we conclude that the volume of the convex hull
of these partitions is also equal to~$C_{n+1}/n!$.  While perhaps not surprising
to the experts in the field~\cite{Bar}, the integer points in these polytopes
have a completely different structure than polytopes themselves.

\subsection{} \label{s:fin-4} \. The following table lists the $f$-vectors
of Tutte polytopes $\pT_n(q,t)$ for $n = 1,\ldots, 10$, $0 \leq q < 1$ and $t > 0$.
The results were obtained using {\tt polymake} (see~\cite{GJ1}).

\smallskip

{\footnotesize
$$\aligned
& 2 \\
& 4,\ 4 \\
& 8, \ 13,\ 7 \\
& 16, \ 37,\ 32, \ 11 \\
& 32, \ 97,\ 117, \ 66, \ 16 \\
& 64, \ 241, \ 375, \ 297, \ 121, \ 22 \\
& 128, \ 577, \ 1103, \ 1130, \ 653, \ 204, \ 29 \\
& 256, \ 1345, \ 3055, \ 3850, \ 2894, \ 1296, \ 323, \ 37 \\
& 512, \ 3073, \ 8095, \ 12130, \ 11255, \ 6597, \ 2381, \ 487, \ 46 \\
& 1024, \ 6913, \ 20735, \ 36050, \ 39865, \ 28960, \ 13766, \ 4117, \ 706, \ 56
\endaligned
$$
}

\noindent
Based on these calculations, we state the following conjecture.
\begin{conj}
For $0 < q < 1$ and $t > 0$, the number of edges of the Tutte polytope $\pT_n(q,t)$ is $3(n-1)2^{n-2}+1$,
and the number of $2$-faces is $2^{n-5} \bigl(9\ts n^2- 29\ts n +38\bigr) -1$.
\end{conj}

\subsection{} \label{s:fin-5} \.
The recurrence relations for inversion polynomials $\Inv_n(t)$ have a long
history, and are used to obtain closed form exponential generating functions 
for~$\Inv_n(t)$.  We refer to~\cite{MR,Ges1,Ges2,GS,Tutte} for several such
results.  The recursive formulas in Theorem~\ref{rec2} are
different, but somewhat similar to those in~\cite{Gil}.

Let us mention that one should not expect to find similar recurrence
relations for general connected graphs, as the problem of computing
(or even approximating) Tutte polynomial $\rT_H(q,t)$ is hard for almost
all values of~$q$ and~$t$~\cite{GJ}.  We refer to~\cite{Welsh} for the 
background and further references.

\subsection{} \label{s:fin-6} \.
The neighbors-first search used in our construction was previously studied
in~\cite{GS} in the context of the Tutte polynomial of a complete graph.
Still, we find its appearance here
somewhat bemusing as other graph traversal algorithms, such
as depth-first search (DFS) and breadth-first search (BFS), are both more standard
in algorithmic literature~\cite{Knuth}. In fact, we learned that it was used
in~\cite{GS} only after much of this work has been finished.

It is interesting to see what happens under graph traversal algorithms
as well.  In the pioneering paper~\cite{GW}, Gessel and Wang showed that the
identity \ts $t^{n-1}\Inv_n(1+t)=F_n(t)$ \ts can be viewed as the result of the
DFS algorithm mapping connected graphs into search trees.  We do not know
what happens for BFS, but surprisingly the algorithm exploring edges of
the graph lexicographically, from smallest to largest, also makes sense.
It was shown by Crapo (in a different language, and for general matroids)
to give internal and external activities~\cite{Crapo}.  In conclusion, let
us mention that BFS, DFS and NFS are special cases of a larger class of
searches known to define combinatorial bijections in a related
setting~\cite{CP}.

\vskip.7cm

\noindent
{\bf Acknowledgements.} \  We are very grateful to Matthias Beck and
Ben Braun for telling us about~\cite{BBL} and the Braun Conjecture,
and to Federico Ardila, Raman Sanyal and Prasad Tetali for helpful
conversations.  The first author was partially supported by
Research Program P1-0297 of the Slovenian Research Agency.
The second author was partially supported by the BSF and NSF grants.


\end{document}